\documentclass{amsart}
\usepackage{amsmath}
\usepackage{amsfonts}
\usepackage{amssymb}
\usepackage{epsfig}
\newtheorem{theorem}{Theorem}[section]

\newtheorem{lm}[theorem]{Lemma}

\newtheorem{pr}[theorem]{Proposition}

\usepackage{color}

\begin{document}
	
\title{Linkage for periplectic supergroups in positive characteristic}
\author{}
\address{}
\email{}
\author{F.Marko and A.N.~Zubkov}
\address{The Pennsylvania State University, 76 University Drive, Hazleton, 18202 PA, USA}
\email{fxm13@psu.edu}
\address{Department of Mathematical Sciences, UAEU, Al-Ain, United Arab Emirates; \linebreak Sobolev Institute of Mathematics, Omsk Branch, Pevtzova 13, 644043 Omsk, Russia}
\email{a.zubkov@yahoo.com}
\begin{abstract}
We consider the periplectic supergroup ${\bf P} (n)$ over a ground field $\Bbbk$ of characteristic $p>2$. We show that there are four blocks of ${\bf P} (n)$ of simple supermodules $L^{\epsilon}(\lambda)$ corresponding to dominant weights
$\lambda$  of even and odd lengths, and the even and odd parity $\epsilon$ of their highest weight vector.
\end{abstract}

\maketitle

\section*{Introduction}
The representation theory of supermodules for the periplectic supergroup ${\bf P} (n)$ over the ground field of characteristic zero was investigated in \cite{chen,many}. The paper \cite{many} contains a detailed analysis of the combinatorics of the category $\mathcal{F}$ of finite-dimensional representations of ${\bf P} (n)$. 
While for basic and queer superalgebras $\mathfrak{g}$, the center of the universal enveloping algebra $\mathcal{U}(\mathfrak{g})$ is large,  the center of the universal enveloping algebra of the periplectic superalgebra $\mathfrak{p}(n)$ is trivial (see \cite{sch}). As a consequence, there are only a few blocks int the category of supermodules over $\mathfrak{g}$.
Indeed, Theorem 9.1.2 of \cite{many} shows that there are $2(n+1)$ blocks of $\mathfrak{p} (n)$. Each block has a representative simple supermodule $L^\epsilon (\lambda)$ determined by the dominant weight $\lambda=\omega_i=i(i-1)\ldots 1 0^{n-i}$ for $i=0, \ldots, n$ and the parity of the highest vector $v_{\lambda}$ of $L(\lambda)$. 

The main goal of the present paper is to investigate the linkage of supermodules for ${\bf P} (n)$ over a field of characteristic $p>2$ and to determine the corresponding blocks of category $\mathcal{F}$.
Changing from a ground field of characteristic zero to a ground field of characteristic $p>2$ brings more nontrivial extensions of simple supermodules and smaller number of blocks. Nevertheless, it is surprising that there are only four blocks for each ${\bf P}(n)$ over a ground field of characteristic $p>2$.

The structure of the paper is as follows. 

Section 1 introduces periplectic supergroups, their root system, and irreducible and induced supermodules.

In Section 2, we use the asymmetry of the root system of ${\bf P} (n)$ to determine a partial case of odd linkage for weights $\lambda$ such that the induced supermodule $H^0_{ev}(\lambda)$, over the maximal even sub-supergroup ${\bf P} (n)_{ev}\sim {\bf GL} (n)$ of ${\bf P} (n)$, is irreducible. 

In Section 3, we define the partial order $\prec$ on dominant weights that behaves well under even linkage and restricted odd linkage.
Then we use restricted odd linkage to show that each dominant weight is linked to one of the simple supermodules $L^{\epsilon}(\lambda)$, where $\lambda=\omega^{a}_{-i}=a^{n-i}(a-1)\ldots (a-i)$, $\epsilon=\pm $ is the parity of the highest weight vector, $0\leq i\leq \min\{n,\frac{p-1}{2}\}$, 
and $a$ is an integer.

In Section 4, we state the criteria for the irreducibility of the even induced module $H^0_{ev}(\lambda)$ and apply it to various special cases used in subsequent sections.

In Section 5, we prove that none of the weights $\omega^a_{-i}$ for $2 \leq i \leq \frac{p-1}{2}$ is minimal concerning 
the order $\prec$ if $p\geq n$. In order to show that $\omega^a_{-i}$ is not minimal with respect to $\prec$, we find weights $\mu$ and $\kappa$ that are even-linked and such that $\omega^a_{-i}$ is linked to $\mu$ via a chain of restricted odd-linked weights, and $\kappa$ is linked to a weight $\nu\precneqq \omega^a_{-i}$ via a chain of restricted odd-linked weights.   

Section 6 is the most technical. In there we prove that none of the weights $\omega^a_{-i}$ for $2 \leq i \leq \frac{p-1}{2}$ is minimal with respect to the order $\prec$ if $p\leq n$. 

That reduces our considerations to weights $\omega^a_{-1}$ and $\omega^a_0$. These weights are investigated in Section 7. Finally, in Theorem \ref{t7.3}, we determine the blocks of ${\bf P} (n)$.  
 
\section{Periplectic supergroup}
We are working over a ground field $\Bbbk$ of zero or odd characteristics.
Let $R$ be a supercommutative superalgebra, ${\bf M}(n|n)(R)$ be a set of block matrices 
$\left( \begin{array}{cc}
X & Y \\
Z & W
\end{array}\right),$ 
where $X,Y,Z,W$ are square matrices of size $n$ with entries in $R$, and 
${\bf GL}(n|n)$ consists of those elements in ${\bf M}(n|n)$ for which $X,W$ are invertible.

Let ${\bf P}(n)$ be a \emph{periplectic} supergroup, where $n\geq 2$. 
Then for any (supercommutative) superalgebra $R$, we have
\[{\bf P}(n)(R)=\{ g\in {\bf GL}(n|n)(R)| ^{st}g J_n g=J_n\}, \]
where
\[ J_n=\left( \begin{array}{cc}
0 & I_n \\
I_n & 0
\end{array}\right), \quad 
g= \left( \begin{array}{cc}
X & Y \\
Z & W
\end{array}\right), 
\text{ and }
^{st} g= \left( \begin{array}{cc}
X^t & Z^t \\
-Y^t & W^t
\end{array}\right).\]
Then $g\in {\bf GL}(n|n)(R)$ belongs to ${\bf P}(n)(R)$ if and only if the following equations are satisfied:
\[Z^t X +X^t Z=0, \ Z^t Y+ X^t W=I_n, \ W^t Y-Y^t W=0, \ W^t X -Y^t Z=I_n. \]
The largest even sub-supergroup ${\bf P}(n)_{ev}$ of ${\bf P}(n)$ consists of matrices
\[ \left( \begin{array}{cc}
X & 0 \\
0 & (X^t)^{-1}
\end{array}\right) \text{ for } X\in {\bf GL}(n).\]
Thus, ${\bf P}(n)_{ev}$  is naturally isomorphic to ${\bf GL}(n)$.

The maximal torus $T$ of ${\bf P}(n)$ consists of matrices
\[ \left( \begin{array}{cc}
X & 0 \\
0 & X^{-1}
\end{array}\right),\]
where $X$ is an invertible diagonal $n\times n$ matrix. We identify $T$ with the standard torus of ${\bf GL} (n)$.

The Lie superalgebra $\mathfrak{p}(n)$ of ${\bf P}(n)$ consists of matrices 
\[ A=\left( \begin{array}{cc}
X & Y \\
Z & W
\end{array}\right)\in {\bf M}(n|n)(\Bbbk)\]
such that $^{st}AJ_n + J_n A=0$. Equivalently, $A$ belongs to $\mathfrak{p}(n)$ if and only if
\[ W=-X^t, \ Z^t=-Z, \ Y^t=Y.  \]

Let $V$ denote the standard ${\bf GL}(n|n)$-supermodule. Fix a homogeneous basis $v_1, \ldots, v_n, v_{\bar 1}, \ldots, v_{\bar n}$ of $V$
such that the parities are given as $|v_i|=0$ and $|v_{\bar i}|=1$ for $1\leq i\leq n$. Let $I$ denote the set $\{1, \ldots, n, {\bar 1}, \ldots, {\bar n} \}$. There is an involution $I\to I$ that sends each $i$ to $\bar i$, and symmetrically, each $\bar i$ to $i$, for $1\leq i\leq n$. We denote this involution by $i\mapsto {\bar i}$ for $i\in I$.  

We define an odd nondegenerate bilinear form $(\  , \ )$ on $V$ by $(v_i, v_j)=\delta_{i, {\bar j}}$ for $i, j\in I$. For any superalgebra $R$, we extend this form to the right $R$-supermodule
$V\otimes R$ by the rule \[(v\otimes a, w\otimes b)_R=(-1)^{|a||w|}(v, w)ab \text{ for }  v, w\in V, a, b\in R, \] where $a, b$ and $v, w$ are initially assumed to be homogeneous, and then the general case follows by the linear extension.

Then ${\bf P}(n)(R)$ consists of all even automorphism of $R$-supermodule $V\otimes R$, preserving the form $( \ , \ )_R$. Additionally, a (homogeneous) operator $A\in {\bf M}(n|n)(\Bbbk)$ belongs to $\mathfrak{p}(n)$ if and only if 
\[ (-1)^{|A||w|}(Av, w)+(v, Aw)=0 \text{ for all } v, w\in V.\]

We fix a basis of ${\bf M}(n|n)(\Bbbk)$ consisting of the matrices $E_{i j}$ with $E_{ij} v_k=\delta_{kj} v_i$ for $ i, j, k\in I$.

The root system of ${\bf P}(n)$ with respect to $T$ is $\Delta=\Delta_0\sqcup\Delta_1$, where 
\[\Delta_0=\{\pm(\epsilon_i-\epsilon_j)| 1\leq i<j\leq n  \} \]
is the subset of \emph{even} roots and
\[\Delta_1=\{\pm(\epsilon_i+\epsilon_j), 2\epsilon_k| 1\leq i< j\leq n, 1\leq k\leq n  \} \]
is the subset of \emph{odd} roots.

The set \[ \Delta^+=\{\epsilon_i-\epsilon_j, -(\epsilon_i+\epsilon_j)| 1\leq i< j\leq n\} \]
is a set of \emph{positive} roots, 
and the complementary set
\[ \Delta^-=\{-(\epsilon_i-\epsilon_j), \epsilon_k+\epsilon_l| 1\leq i< j\leq n, 1\leq k\leq l\leq n\} \]
is the set of \emph{negative} roots.
Then $\Delta=\Delta^+\sqcup\Delta^-$ but this decomposition is asymmetric and $-\Delta^+\neq \Delta^-$.
We denote by $B^+$ and $B^-$ the Borel sub-supergroups of ${\bf P}(n)$, and by  $U^+$ and $U^-$ the unipotent sub-supergroups corresponding to $\Delta^+$ and 
$\Delta^-$, respectively.
Additionally, we denote $B^+_{ev}=B^+\cap {\bf P}(n)_{ev}$ and $B^-_{ev}=B^-\cap {\bf P}(n)_{ev}$.
Furthermore, define $\rho=(\frac{n}{2}, \frac{n-2}{2}, \ldots, \frac{2-n}{2}, \frac{-n}{2})$ as a half-sum of all positive roots in ${\bf P}(n)_{ev}\simeq{\bf GL}(n)$. 
For a root $\alpha$ of ${\bf GL}(n)$, denote by $\alpha^{\vee}$ its dual root, and denote by $\langle .,.\rangle$ the bilinear form such that $\langle \epsilon_i,\epsilon_j^{\vee}\rangle=\delta_{ij}$. 
The affine Weyl group $W$ of ${\bf GL}(n)$ is generated by reflections $s_{\beta, kp}$ given by
\[s_{\beta, kp}(\lambda)=\lambda-(\langle \lambda, \beta^{\vee}\rangle -kp)\beta\]
for $\beta\in \Delta_0\cap \Delta^+$ and integers $k$.
The dot action of the affine Weyl group $W$ on weights $\lambda$ is given as 
$w\bullet \lambda=w(\lambda+\rho)-\rho$. 
The affine reflection under the dot action is given as
\[s_{\beta,kp}\bullet \lambda=\lambda-(\langle \lambda+\rho,\beta^{\vee}\rangle-kp)\beta\]
or if $\beta=\epsilon_i-\epsilon_j$, then 
\[s_{\epsilon_i-\epsilon_j, kp}\bullet \lambda = \lambda -(\lambda_i-\lambda_j-i+j-kp)(\epsilon_i-\epsilon_j).\]

Let $P^+$ denote the parabolic sub-supergroup of ${\bf P}(n)$ consisting of all matrices
 \[ \left( \begin{array}{cc}
 X & 0 \\
 Z & (X^t)^{-1}
 \end{array}\right) \]
such that $X^t Z=-Z^t X$, and $P^-$ denote the parabolic sub-supergroup of ${\bf P}(n)$ consisting of all matrices
\[ \left( \begin{array}{cc}
X & Y \\
0 & (X^t)^{-1}
\end{array}\right) \]
such that $XY^t=YX^t$.  
The natural supergroup morphisms $P^+\to {\bf GL}(n)$ defined as 
\[  \left( \begin{array}{cc}
X & 0 \\
Z & (X^t)^{-1}
\end{array}\right)\mapsto X, \ \mbox{and} \ \left( \begin{array}{cc}
X & 0 \\
Z & (X^t)^{-1}
\end{array}\right)\mapsto (X^t)^{-1}, \]
and natural supergroup morphisms $P^-\to {\bf GL}(n)$ defined as
\[\left( \begin{array}{cc}
X & Y \\
0 & (X^t)^{-1}
\end{array}\right)\mapsto X, \ \mbox{and} \ \left( \begin{array}{cc}
X & Y \\
0 & (X^t)^{-1}
\end{array}\right)\mapsto (X^t)^{-1}   \]
 are split. Therefore, $P^+\simeq {\bf GL}(n)\ltimes U^+$ and $P^-\simeq {\bf GL}(n)\ltimes U^-$, where the \emph{purely-odd unipotent} sub-supergroups $U^+$ and $U^-$ of $P^+$ and $P^-$  consist of the matrices
\[\left( \begin{array}{cc}
I_n & 0 \\
Z & I_n
\end{array}\right) \text{ such that } Z=-Z^t,\]
and
\[ \left( \begin{array}{cc}
I_n & Y \\
0 & I_n
\end{array}\right) \text{ such that } Y=Y^t, \]
respectively. 
Also, $B^+$ is a sub-supergroup of $P^+$, and $B^-$ is a sub-supergroup of $P^-$. 
\begin{lm}\label{decomposition}
There is a commutative diagram whose arrows are superscheme isomorphisms.
\[\begin{array}{ccccc}
 & & {\bf P}(n) & & \\
& \swarrow  & & \searrow & \\
P^+\times U^- & & &  & U^+\times P^- \\
& \searrow & &  \swarrow & \\
 & & U^+\times {\bf GL}(n)\times U^-& &  
 \end{array}.\]
\end{lm}
\begin{proof}
The top maps are defined as
 \[\left( \begin{array}{cc}
 X & Y \\
 Z & W
 \end{array}\right)\mapsto \left( \begin{array}{cc}
 X & 0 \\
 Z & (X^t)^{-1}
 \end{array}\right)\times \left( \begin{array}{cc}
 I_n & X^{-1}Y \\
 0 & I_n
 \end{array}\right)  \]
 and
 \[ \left( \begin{array}{cc}
 X & Y \\
 Z & W
 \end{array}\right)\mapsto \left( \begin{array}{cc}
 I_n & 0 \\
 ZX^{-1} & I_n
 \end{array}\right)\times \left( \begin{array}{cc}
 X & Y \\
 0 & (X^t)^{-1}
 \end{array}\right),\]
 respectively. The bottom maps are defined as
 \[ \left( \begin{array}{cc}
 X & 0 \\
 Z & (X^t)^{-1}
 \end{array}\right)\to \left( \begin{array}{cc}
 I_n & 0 \\
 ZX^{-1} & I_n
 \end{array}\right)\times \left( \begin{array}{cc}
 X & 0 \\
 0 & (X^t)^{-1}
 \end{array}\right)\]	
 and
 \[ \left( \begin{array}{cc}
 X & Y \\
 0 & (X^t)^{-1}
 \end{array}\right)\mapsto \left( \begin{array}{cc}
 X & 0 \\
 0 & (X^t)^{-1}
 \end{array}\right)\times \left( \begin{array}{cc}
 I_n & X^{-1}Y \\
 0 & I_n
 \end{array}\right), \]
 respectively.
\end{proof}

The following is the symmetric version of the above lemma. 
\begin{lm}\label{symmetric}
There is a commutative diagram whose arrows are superscheme isomorphisms.
\[\begin{array}{ccccc}
& & {\bf P}(n) & & \\
& \swarrow  & & \searrow & \\
P^-\times U^+ & & &  & U^-\times P^+ \\
& \searrow & &  \swarrow & \\
& & U^-\times {\bf GL}(n)\times U^+& &  
\end{array}.\]
\end{lm}
\begin{proof}
The top maps are defined as
\[\left( \begin{array}{cc}
X & Y \\
Z & W
\end{array}\right)\mapsto \left( \begin{array}{cc}
(W^t)^{-1} & Y \\
0 & W
\end{array}\right)\times \left( \begin{array}{cc}
I_n & 0 \\
W^{-1}Z & I_n
\end{array}\right)\]
and
\[\left( \begin{array}{cc}
X & Y \\
Z & W
\end{array}\right)\mapsto \left( \begin{array}{cc}
I_n & YW^{-1} \\
0 & I_n
\end{array}\right)\times \left( \begin{array}{cc}
(W^t)^{-1} & 0 \\
Z & W
\end{array}\right)\]	
respectively. The bottom maps are defined as 
\[\left( \begin{array}{cc}
(W^t)^{-1} & Y \\
0 & W
\end{array}\right)\mapsto \left( \begin{array}{cc}
I_n & YW^{-1} \\
0 & I_n
\end{array}\right)\times \left( \begin{array}{cc}
(W^t)^{-1} & 0 \\
0 & W
\end{array}\right) \]
and 
\[ \left( \begin{array}{cc}
(W^t)^{-1} & 0 \\
Z & W
\end{array}\right)\mapsto \left( \begin{array}{cc}
(W^t)^{-1} & 0 \\
0 & W
\end{array}\right)\times \left( \begin{array}{cc}
I_n & 0 \\
W^{-1} Z & I_n
\end{array}\right) \]
respectively.
\end{proof}

We denote the \emph{standard coordinate functions} on ${\bf P}(n)$ by 
\[x_{ij}, y_{i {\bar j}}, z_{{\bar i} j}, w_{{\bar i} {\bar j}} \text{ for } 1\leq i, j\leq n.  \]
Set
\[{\bf X}=(x_{ij})_{1\leq i, j\leq n}, \ {\bf Y}=(y_{i {\bar j}})_{1\leq i, j\leq n}, \ {\bf Z}=(z_{{\bar i} j})_{1\leq i, j\leq n}, \ {\bf W}=(w_{{\bar i} {\bar j}})_{ 1\leq i, j\leq n}.\]
Denote ${\bf Y}'={\bf X}^{-1}{\bf Y}$, ${\bf Z}'={\bf Z}{\bf X}^{-1}$, and 
${\bf Y}''={\bf Y}{\bf W}^{-1}$, ${\bf Z}''={\bf W}^{-1}{\bf Z}$. 

The weights $\lambda$ on ${\bf P}(n)$ are the weights of ${\bf P}(n)_{ev}\simeq {\bf GL}(n)$.
We are using two orders on the set of weights of ${\bf P} (n)$, the \emph{dominant} and \emph{anti-dominant} ones. The dominant order is defined as
\[\lambda\geq\mu \ \mbox{if and only if} \ \lambda-\mu\in\sum_{\alpha\in\Delta^+}\mathbb{Z}_{\geq 0}\alpha,\]
and the anti-dominant order is defined symmetrically, that is
\[\lambda\geq^a\mu \ \mbox{if and only if} \ \lambda-\mu\in\sum_{\alpha\in\Delta^-}\mathbb{Z}_{\geq 0}\alpha . \]

According to \cite[Proposition 4.11 and Theorem 5.5]{tshib}, each simple ${\bf P} (n)$-supermodule is uniquely defined by its highest weight $\lambda$ with respect to $\geq$ and the parity $\epsilon$ of its highest weight vector, and we denote this supermodule by $L^{\epsilon}_{B^-}(\lambda)$. We will write $L(\lambda)$ for $L^+_{B^-}(\lambda)$. The weight $\lambda$ is \emph{dominant}, i.e. $\lambda_1\geq\lambda_2\geq\ldots\geq\lambda_n$, and the weight space $L(\lambda)_{\lambda}=L(\lambda)^{U^+}$ is a one-dimensional superspace. Denote by $\Pi$ the parity shift functor.

Let $H^0_{B^{\pm}}(\lambda)=\mathrm{ind}^{{\bf P} (n)}_{B^{\pm}} \Bbbk_{\lambda}$ denote the induced supermodules, and 
$H^0_{B^{\pm}_{ev}}(\lambda)=\mathrm{ind}^{{\bf GL}(n)}_{B^{\pm}_{ev}} \Bbbk_{\lambda}$ the induced supermodules. 
Further, denote by $V_{B_{ev}^{\mp}}(-\lambda)\simeq H^0_{B^{\pm}_{ev}}(\lambda)^*$ the corresponding Weyl supermodules. 
Each irreducible supermodule (up to parity shift) is isomorphic to the socle of the induced supermodule $H^0_{B^-}(\lambda)$ for some dominant weight $\lambda$. Conversely, if $H^0_{B^-}(\lambda)\neq 0$, then $\lambda$ is dominant and the socle of $H^0_{B^-}(\lambda)$ is $L(\lambda)$. 

Symmetrically, every simple ${\bf P} (n)$-supermodule is uniquely defined by the highest weight $\mu$ with respect to $\leq^a$ and the parity $\epsilon$ of its highest weight vector, and we denote it by $L^{\epsilon}_{B^+}(\mu)$. The weight $\mu$ is \emph{anti-dominant}, i.e. $-\mu$ is dominant. As above, $L_{B^+}(\mu)_{\mu}=L_{B^+}(\mu)^{U^-}$ is one-dimensional. Moreover, $L_{B^+}(\mu)$ is isomorphic to the socle of $H^0_{B^+}(\lambda)$.

The degree of a weight $\mu$ is $|\mu|=\sum_{i=1}^n \mu_i$.
For a supermodule $M$, we have a degree decomposition $M=\oplus M_k$, where the degree-$k$-component $M_k$ of $M$ is the direct sum of weightspaces $M_\mu$ for all weights $\mu$ such that $|\mu|=k$. 
Each simple supermodule $L(\lambda)$ has the smallest nonzero degree component $l=|\lambda|$.  
Its unique largest nonzero degree component is of degree $l+2r$, where $r\leq\frac{n(n-1)}{2}$.
Then  $L(\lambda)\simeq \Pi^r L_{B^+}(\lambda^+)$ for an anti-dominant weight $\lambda^+$, uniquely defined by $\lambda$.

Let $L_{ev}(\lambda)$ be the irreducible ${\bf P} (n)_{ev}\simeq {\bf GL}(n)$-module of the highest weight $\lambda$. 

If $H$ is a sub-supergroup of ${\bf P} (n)$, then we denote by $Dist(H)$ the distribution algebra of $H$.

\begin{pr}\label{l and l+2r components, minimal weight}
We have 
\begin{enumerate}		
\item $L(\lambda)_{l+2r}\simeq \Pi^r L_{B^+_{ev}}(\lambda^+)$;
\item $L(\lambda)^*\simeq \Pi^r L(-\lambda^+)$;
\item $L(\lambda)_l\simeq L_{ev}(\lambda)$;
\end{enumerate}
\end{pr}
\begin{proof}
Let $v\in L(\lambda)_{l+2r}$ be a homogeneous $B^-_{ev}$-primitive element of a weight $\nu$. Since $Dist(B^-)=Dist(U^-)Dist(B^-_{ev})$ and
$Dist(U^-)$ acts on $v$ trivially due the maximality of $r$, we infer that $v$ is $B^-$-primitive. Therefore,  $v$ is the highest weight vector of
$L(\lambda)$ with respect to the anti-dominant order, which implies $\nu=\lambda^+$. Moreover, $v$ is unique (up to a nonzero scalar). Hence it generates the unique simple ${\bf P} (n)_{ev}$-sub-supermodule of $L(\lambda)_{l+2r}$, that is isomorphic to $L_{B^+_{ev}}(\lambda^+)$. Besides, 
the parity of $L(\lambda)_{l+2r}$ coincides with the parity of $v$, which is $(-1)^r$. Finally, an element from $L(\lambda)=Dist({\bf P} (n))v=Dist(B^+)v$ belongs to $L(\lambda)_{l+2r}$ if and only if it belongs to $Dist(B^+_{ev})v\subseteq Dist({\bf P} (n)_{ev})v$. Indeed, $Dist(B^+)=Dist(U^+)Dist(B^+_{ev})$ and any
non-scalar element of $Dist(U^+)$ maps any element from $L(\lambda)_{l+2r}$ to a component of degree $l+2k$ with $k< r$. This proves (1).

For any dominant weight $\mu$, the degree component $L(\mu)_{|\mu|}$ has $L_{ev}(\mu)=L_{B^-_{ev}}(\mu)$ as a composition factor. Since
\[(L(\lambda)_{l+2r})^*=L(\lambda)^*_{-l-2r}\simeq \Pi^r L_{B^+_{ev}}(\lambda^+)^*\simeq \Pi^r L_{B^+_{ev}}(-w_0(\lambda^+))=\Pi^r L_{B^-_{ev}}(-\lambda^+)\]
and $L(\lambda)^*_{-l-2r}$ is the component of $L(\lambda)^*$ of minimal degree, the statement (2) follows.	

Since the correspondence $\lambda\mapsto-\lambda^+$ on the set of dominant weights is one-to-one, we obtain the statement (3).
\end{proof}

\begin{lm}\label{inflation for P^-}
The ${\bf GL} (n)$-module $H^0_{ev}(\lambda)$ has the natural structure of a $P^-$-supermo-dule via the epimorphism 
$P^-\to {\bf GL}(n)$ given via 
\[\left( \begin{array}{cc}
(W^t)^{-1} & Y \\
0 & W
\end{array}\right)\mapsto W . \] Moreover, as such a $P^-$-supermodule it is isomorphic to $\mathrm{ind}^{P^-}_{B^-} \Bbbk_{\lambda}$.
\end{lm}
\begin{proof}
Since $P^-\simeq {\bf GL}(n)\ltimes U^-$ and $B^-\simeq B^-_{ev}\ltimes U^-$, we can apply \cite[Theorem 10.1 and Corollary 10.2]{zub1}.	
\end{proof}
Slightly abusing notations, we denote the coordinate functions on $P^-$ by the same symbols $y_{i {\bar j}}$ and $w_{{\bar i} {\bar j}}$ as the coordinate functions on ${\bf P}(n)$.
Then $\Bbbk [w_{{\bar i} {\bar j}}\mid 1\leq i, j\leq n]_{\det({\bf W})}$ is a Hopf sub-superalgebra of $\Bbbk [P^-]$, which is naturally isomorphic to $\Bbbk [{\bf GL}(n)]$. Moreover, the epimorphism of Lemma \ref{inflation for P^-} 
is dual to the embedding $\Bbbk [w_{{\bar i} {\bar j}}\mid 1\leq i, j\leq n]_{\det({\bf W})}\to \Bbbk [P^-]$. In particular, since $H^0_{ev}(\lambda)$ can be naturally embedded into $\Bbbk [{\bf GL}(n)]$ as a right $\Bbbk [{\bf GL}(n)]$-module, it is also naturally embedded into $\Bbbk [P^-]$ as a right $\Bbbk [P^-]$-supercomodule, that is isomorphic to $\mathrm{ind}^{P^-}_{B^-} \Bbbk_{\lambda}$. In what follows we identify $\mathrm{ind}^{P^-}_{B^-} \Bbbk_{\lambda}$ with this sub-supercomodule of $\Bbbk [P^-]$. 

\begin{lm}\label{main isomorphism}
There is a superspace isomorphism
\[H^0_{B_{ev}^-}(\lambda)\otimes \Bbbk [U^+]\to H^0_{B^-}(\lambda)\subseteq \Bbbk [{\bf P}(n)],\]
induced by the map which is the identity on $\bf W$ and sends $\bf Z$ to ${\bf Z}''={\bf W}^{-1}{\bf Z}$.	
\end{lm}
\begin{proof}
Combine the canonical isomorphism $H^0_{B^-}(\lambda)\simeq \mathrm{ind}^{{\bf P}(n)}_{P^-}\mathrm{ind}^{P^-}_{B^-} \Bbbk_{\lambda}$
with \cite[Lemma 6.1]{marzub}.
\end{proof}
Note that, in the above Lemma 
the ${\bf GL} (n)$-module $H^0_{ev}(\lambda)$ is identified with $ind^{{\bf GL}(n)}_{B^+_{ev}} \Bbbk_{-\lambda}$.
In fact, we identify ${\bf P}(n)_{ev}\simeq {\bf GL}(n)$ with the block $W$ of $P^-$. Therefore, $B^-_{ev}$ is identified with the Borel subgroup of ${\bf GL} (n)$ consisting of all UPPER triangular matrices. Moreover, the one-dimensional $B^-$-supermodule $\Bbbk_{\lambda}$ has the weight $\lambda$ concerning the action of the $X$-block of the maximal torus $T$. The $W$-block of $T$ acts via the weight $-\lambda$.

\begin{lm}\label{main2}
There is a superspace isomorphism
\[H^0_{B_{ev}^+}(\lambda)\otimes \Bbbk [U^-]\to H^0_{B^+}(\lambda)\subseteq \Bbbk [{\bf P}(n)],\]
induced by a map which is the identity on $\bf X$ and maps $\bf Y$ to ${\bf Y}'={\bf X}^{-1}{\bf Y}$.	
\end{lm}
\begin{proof}
Combine the canonical isomorphism $H^0_{B^+}(\lambda)\simeq \mathrm{ind}^{{\bf P}(n)}_{P^+}\mathrm{ind}^{P^+}_{B^+} \Bbbk_{\lambda}$
with \cite[Lemma 6.1]{marzub}.
\end{proof}

\begin{lm}\label{goodfiltration}
The supermodule $H^0_{B^-}(\lambda)_{l+2}$ has a good filtration with factors
$H^0_{B^-_{ev}}(\lambda+\epsilon_i+\epsilon_j)$ for $1\leq i < j\leq n$ such that $\lambda+\epsilon_i+\epsilon_j$ is dominant.

The module $H^0_{B^+}(\lambda^+)_{l+2r-2}$ has a good filtration with the factors
$H^0_{B^+_{ev}}(\lambda^+ -\epsilon_i-\epsilon_j)$ for $1\leq i\leq j\leq n$ such that $\lambda^+-\epsilon_i-\epsilon_j$ is anti-dominant and $\lambda^+_i\neq \lambda^+_j$ if $j=i+1$.
\end{lm}
\begin{proof}
By Lemma \ref{main isomorphism}, $H_{B^-}^0(\lambda)$, regarded as a sub-supermodule of $\Bbbk[{\bf P} (n)]$, can be identified with  $H^0_{B^-_{ev}}(\lambda)\otimes \Bbbk[{\bf Z}'']$, where $H^0_{B^-_{ev}}(\lambda)$ is identified with a sub-supermodule of $\Bbbk[{\bf P} (n)_{ev}]\simeq\Bbbk[{\bf W}]_{\det({\bf W})}$. 

The matrix ${\bf Z}''$ is skew-symmetric and ${\bf GL} (n)$ acts on its entries as
${\bf Z}''\mapsto g^t {\bf Z}''g$ for $g\in {\bf GL}(n)$. In fact, ${\bf Z}''={\bf W}^{-1}Z$ and ${\bf GL} (n)$ acts on the blocks as
\[{\bf X}\mapsto {\bf X}g, {\bf Y}\mapsto {\bf Y}(g^t)^{-1},  {\bf Z}\mapsto {\bf Z}g, {\bf W}\mapsto {\bf W}(g^t)^{-1}.  \]

Thus the space spanned by the entries of ${\bf Z}''$ is a ${\bf GL} (n)$-module that is isomorphic to $\Lambda^2(V)=H^0_{B^-_{ev}}(\epsilon_1+\epsilon_2)$,
where $V$ is the standard ${\bf GL} (n)$-module of dimension $n$.
Moreover, $\Bbbk[{\bf Z}'']$ is isomorphic to $\Lambda(\Lambda^2(V))$ and
\[H^0_{B^-}(\lambda)_{l+2k}\simeq H^0_{B^-_{ev}}(\lambda)\otimes\Lambda^{k}(\Lambda^2(V)) \text{ for } 0\leq k\leq \frac{n(n-1)}{2}.\]
Using Donkin-Mathieu theorem, Pieri's rule and comparing formal characters, we conclude that $H^0_{B^-}(\lambda)_{l+2}$ has good filtration with factors
$H^0_{B^-_{ev}}(\lambda+\epsilon_i+\epsilon_j)$, where $1\leq i < j\leq n$ and $\lambda+\epsilon_i+\epsilon_j$ is dominant.

For the supermodule $H^0_{B^+}(\mu)$, we use Lemma \ref{main2}. The matrix ${\bf Y}'$ is symmetric and ${\bf GL} (n)$ acts on its entries as ${\bf Y}'\mapsto g^{-1}{\bf Y}'(g^t)^{-1}$ for $g\in {\bf GL}(n)$. Thus, the ${\bf GL} (n)$-submodule spanned by the entries of ${\bf Y}'$ is isomorphic to $S^2(V)^*\simeq H^0_{B^-_{ev}}(2\epsilon_1)^*\simeq V_{B^+_{ev}}(-2\epsilon_1)$. 

For any $0\leq k\leq \frac{n(n+1)}{2}$, there is a ${\bf GL} (n)$-module isomorphism
\[H^0_{B^+}(\lambda^+)_{l+2r-2k}\simeq H^0_{B^+_{ev}}(\lambda^+)\otimes\Lambda^k(H^0_{B^-_{ev}}(2\epsilon_1)^*).  \]
In particular, we have
\[H^0_{B^+}(\lambda^+)_{l+2r-2}\simeq H^0_{B^+_{ev}}(\lambda^+)\otimes H^0_{B^-_{ev}}(2\epsilon_1)^*= V_{B^-_{ev}}(-\lambda^+)^*\otimes 
H^0_{B^-_{ev}}(2\epsilon_1)^*\simeq \]
\[(V_{B^-_{ev}}(-\lambda^+)\otimes H^0_{B^-_{ev}}(2\epsilon_1))^*.   \]
If  $char \ \Bbbk > 2$, then $S^2(V)\simeq H^0_{B^-_{ev}}(2\epsilon_1)$ is a tilting module, hence $H^0_{B^-_{ev}}(2\epsilon_1)=V_{B^-_{ev}}(2\epsilon_1)$. Thus
\[V_{B^-_{ev}}(-\lambda^+)\otimes H^0_{B^-_{ev}}(2\epsilon_1)=V_{B^-_{ev}}(-\lambda^+)\otimes V_{B^-_{ev}}(2\epsilon_1) \]
has a Weyl filtration with the factors $V_{B^-_{ev}}(-\lambda^+ +\epsilon_i+\epsilon_j)$, where $1\leq i\leq j\leq n$, $-\lambda^+ +\epsilon_i+\epsilon_j$ is dominant and if $j=i+1$, then $\lambda^+_i <\lambda^+_j$. Therefore, 
$H^0_{B^+}(\lambda^+)_{l+2r-2}$ has a good filtration with the corresponding factors
$H^0_{B^+_{ev}}(\lambda^+ -\epsilon_i-\epsilon_j)$ for $1\leq i\leq j\leq n$.
\end{proof}

We say that weights $\lambda,\mu$ of ${\bf P} (n)$ are linked, and write $\lambda\sim \mu$, if there is a sequence of weights $\lambda=\kappa_0,\kappa_1, \ldots, \kappa_s=\mu$ such that for each consequtive weights $\kappa_i$ and $\kappa_{i+1}$ there is an indecomposable ${\bf P}(n)$-supermodule which has composition factors
$L^{\epsilon_i}(\kappa_i)$ and $L^{\epsilon_{i+1}}(\kappa_{i+1})$ for some $\epsilon_i,\epsilon_{i+1}=\pm$.

The even linkage of weights $\lambda,\mu$ of ${\bf P} (n)_{ev}={\bf GL}(n)$ is denoted by $\lambda\sim_{ev} \mu$
and defined analogously. We have $\lambda\sim_{ev}\lambda$ if there is a sequence of weights $\lambda=\kappa_0,\kappa_1, \ldots, \kappa_s=\mu$ such that for each consequtive weigths $\kappa_i$ and $\kappa_{i+1}$ there is an indecomposable ${\bf GL}(n)$-module which has $L_{ev}(\kappa_i)$ and $L_{ev}(\kappa_{i+1})$ as its composition factors.

The even linkage for ${\bf GL}(n)$ was described in \cite{donkin} as follows.
A weight $\lambda$ has defect $d(\lambda)$ if $d(\lambda)$ is the largest nonnegative integer $d$ such that 
$\lambda_i-\lambda_{i+1}\equiv -1\pmod {p^d}$ for every $1\leq i\leq n$. 
Then $\lambda\sim_{ev}\mu$ if and only if $d(\lambda)=d(\mu)=d$ and there is a permutation $\sigma\in \Sigma_n$ such that $\lambda_i-i\equiv \mu_{\sigma(i)}-\sigma(i) \pmod {p^{d+1}}$ 
for each $i=1, \ldots, n$.

\section{A partial case of odd linkage when $H^0_{B^-_{ev}}(\lambda)$ is irreducible}

In this section we assume that $\lambda$ is a dominant weight such that $H^0_{B^-_{ev}}(\lambda)=L_{B_{ev}^-}(\lambda)$. This assumption is satisfied for all dominant weights if $char \ \Bbbk=0$.
In what follows, we will write $L_{ev}(\lambda)$ instead of $L_{B_{ev}^-}(\lambda)$.

If $V$ is a supermodule, we denote by $[V]$ its image in the Grothendieck group for the group ${\bf P} (n)_{ev}={\bf GL}(n)$.

We will require the following statement. 

\begin{lm}\label{dual}
Let $\kappa$ be a dominant weight. Then 
$[(H^0_{B^+_{ev}}(-\kappa))^*]=[H^0_{B^-_{ev}}(\kappa)]$.
\end{lm}
\begin{proof}
The claim follows from $(H^0_{B^+_{ev}}(-\kappa))^*=V_{B^-_{ev}}(\kappa)$ and $[V_{B^-_{ev}}(\kappa)]=
[H^0_{B^-_{ev}}(\kappa)]$.
\end{proof}

\begin{lm}\label{even linkage}
If $\lambda\sim_{ev}\mu$, then $\lambda\sim \mu$.
\end{lm}
\begin{proof}
Using the strong linkage principle for ${\bf GL}(n)$ and symmetry, we can assume that $L_{ev}(\mu)$ is a composition factor of $H^0_{B^-_{ev}}(\lambda)$. Let $M$ be a submodule of $H^0_{B^-_{ev}}(\lambda)$ whose top is isomorphic to $L_{ev}(\mu)$. Since ${\bf P}(n)/P^-$ is an affine (purely odd) superscheme, the functor $\mathrm{ind}^{{\bf P}(n)}_{P^-}$ is exact, and Lemma \ref{main isomorphism} implies that $H^0_{B^-}(\lambda)$ contains a sub-supermodule a factor of which is isomorphic to $\mathrm{ind}^{{\bf P}(n)}_{P^-} L_{ev}(\mu)\subseteq H^0_{B^-}(\mu)$. Thus $L(\mu)$ is a composition factor of $H^0_{B^-}(\lambda)$.
\end{proof}

\begin{pr}\label{main}
Assume $\lambda$ is a dominant weight such that $H^0_{B^-_{ev}}(\lambda)=L_{ev}(\lambda)$.
If $\lambda+2\epsilon_j$ is dominant for $1\leq j\leq n$, then 
$\lambda\sim \lambda+2\epsilon_j$.
If $1\leq i<n$ is such that $\lambda_i=\lambda_{i+1}$ and $\lambda+\epsilon_i+\epsilon_{i+1}$ is dominant, then 
$\lambda \sim \lambda+\epsilon_i+\epsilon_{i+1}$.
\end{pr}
\begin{proof} 
First, we prove the statements under the assumption that 
$L_{ev}(\lambda+2\epsilon_j)$ does not appear as a composition factor of any $H^0_{B^-_{ev}}(\lambda+\epsilon_i+\epsilon_{i+1})$ such that $\lambda_i=\lambda_{i+1}$; and if $\lambda_i=\lambda_{i+1}$ then $L_{ev}(\lambda+\epsilon_i+\epsilon_{i+1})$ does not appear as a composition factor of any $H^0_{B_{ev}^-}(\lambda+2\epsilon_j)$. This assumption will be removed later.

Let $\lambda$ be a dominant weight of length $l$, and $\pi=-\lambda^+$. Then $L(\lambda)^*\simeq L(\pi)$.
Since $L(\pi)^*\simeq L(\lambda)$, 
the highest anti-dominant weight $\pi^+$ of $L(\pi)$ satisfies $\pi^+=-\lambda$. 
Since $L(\pi)\simeq L_{B^+}(\pi^+)$ is the socle of $H_{B^+}^0(\pi^+)$, by taking the duals we obtain that $L(\lambda)=L(\pi)^*$ is a factorimage that is the top 
of $(H^0_{B^+}(\pi^+))^*=M$.

From $H^0_{B^+}(\pi^+)= \oplus_{i=0}^{n^+} H^0_{B^+}(\pi^+)_{-l-2i}$, we get 
$M= \oplus_{i=0}^{n^+} M_{l+2i}$. Since $\pi^+=-\lambda$, using Lemma \ref{dual}, we have 
\[[(H^0_{B^+_{ev}}(\pi^+-\epsilon_i-\epsilon_j))^*]=[H^0_{B^-_{ev}}(\lambda+\epsilon_i+\epsilon_j)].\] 
It follows from Lemma \ref{goodfiltration} that 
\[[M_{l+2}]=\sum [H^0_{B^-_{ev}}(\lambda+\epsilon_i+\epsilon_j)],\]
where the sum is over all $1\leq i\leq j\leq n$ such that $\lambda+\epsilon_i+\epsilon_j$ is dominant and $\lambda_i\neq \lambda_j$ if $j=i+1$.

On the other hand, Lemma \ref{goodfiltration} also gives
\[[H^0_{B^-}(\lambda)_{l+2}]=\sum_{1\leq i<j\leq n} [H^0_{B^-_{ev}}(\lambda+\epsilon_i+\epsilon_j)],\]
where the sum is over dominant weights $\lambda+\epsilon_i+\epsilon_j$.

Assume that $\lambda$ is such that $\lambda_i=\lambda_{i+1}$ for some $i$ and $L_{ev}(\lambda+\epsilon_i+\epsilon_{i+1})$
does not appear as a composition factor in any of $H^0_{B^-_{ev}}(\lambda+2\epsilon_j)$ (This is automatically satisfied if $char \ \Bbbk=0$.)  Then the multiplicity of $[L_{ev}(\lambda+\epsilon_i+\epsilon_{i+1}]$ in $[H^0_{B^-}(\lambda)_{l+2}]$ is bigger than its multiplicity in $[M_{l+2}]$, hence in $[L(\lambda)_{l+2}]$. Therefore,  the simple supermodule $L(\lambda+\epsilon_i+\epsilon_{i+1})$
appears as a composition factor of $(H^0_{B^-}(\lambda)/L(\lambda))_{l+2}$.
Therefore, there is a vector $w$ of weight $\kappa\sim_{ev}\lambda+\epsilon_i+\epsilon_{i+1}$ that is one of the generators of the socle of the supermodule
$H^0_{B^-}(\lambda)_{l+2}/L(\lambda)_{l+2}$. Then, the assumption $H^0_{B^-}(\lambda)_l=L(\lambda)_l$ implies that $w$ is primitive modulo $L(\lambda)$, meaning that $Dist(U^-)w \subset L(\lambda)$. Thus $L(\kappa)$ is a composition 
factor of $H^0_{B^-}(\lambda)$ and $\kappa\sim \lambda$.
Since $\lambda+\epsilon_i+\epsilon_{i+1}\sim_{ev}\kappa\sim \lambda$, we conclude that
$\lambda\sim \lambda+\epsilon_i+\epsilon_{i+1}$.

Assume $L_{ev}(\lambda+2\epsilon_j)$ does not appear as a composition factor in any of $H^0_{B^-_{ev}}(\lambda+\epsilon_i+\epsilon_{i+1})$ such that $\lambda_i=\lambda_{i+1}$ (This is automatically satisfied if $char \ \Bbbk=0$.)
Then the multiplicity of $[L_{ev}(\lambda+2\epsilon_j)]$ in $[M_{l+2}]$ is bigger than its multiplicity in $[H^0_{B^-}(\lambda)_{l+2}]$, hence in $[L(\lambda)_{l+2}]$. 
Since $[M_l]=[(H^0_{B^+_{ev}}(\pi^+))^*]=[H^0_{B^-_{ev}}(\lambda)]$
and by the assumption of the Proposition there is $[H^0_{B^-_{ev}}(\lambda)]=[L(\lambda)_l]$, we obtain $M_l=L(\lambda)_l$, which is equivalent to $rad(M_l)=0$.
Therefore, there is a vector $w$ of weight $\kappa\sim_{ev}\lambda+2\epsilon_j$ that is one of the generators of the socle of the module
$M_{l+2}$. Then, $M_l=L(\lambda)_l$ implies that 
$w$ is primitive, meaning that $Dist(U^-)w =0$. Thus $L(\kappa)$ is a composition 
factor of $M$ and $\kappa\sim \lambda$.
Since $\lambda+2\epsilon_j\sim_{ev}\kappa\sim \lambda$, we conclude that
$\lambda\sim \lambda+2\epsilon_j$.

Finally, denote by $S$ the set of all dominant weights of type $\lambda+2\epsilon_j$ and $\lambda+\epsilon_i+\epsilon_{i+1}$ such that $\lambda_i=\lambda_{i+1}$.  Now consider the case when a weight $\lambda+2\epsilon_j$ is linked to some $\lambda+\epsilon_i+\epsilon_{i+1}$ such that $\lambda_i=\lambda_{i+1}$. Denote by $B$ a block that contains all weights from $S$ that are even-linked to 
such $\lambda+2\epsilon_j$. Then there is a weight $\kappa$ from $B\cap S$ such that $L_{ev}(\kappa)$ does not appear as a composition factor in any other $H^0_{ev}(\sigma)$ for $\sigma\in B\cap S$ different from $\kappa$. 
By the previous considerations, $\kappa\sim \lambda$. Then for every $\sigma\in B\cap S$, we have 
$\sigma\sim_{ev}\kappa\sim \lambda$, thus $\lambda\sim \sigma$. Repeating this argument for all blocks $B$ as above, we conclude that 
$\lambda\sim \lambda+2\epsilon_j$.
Considerations for $\lambda\sim\lambda+\epsilon_i+\epsilon_{i+1}$ whenever $\lambda_i=\lambda_{i+1}$ 
are analogous.
\end{proof}

Proposition \ref{main} is our main tool for establishing (odd) linkage of weights. In what follows, we use Proposition \ref{main} without explicit reference to establish this partial odd linkage of weights.

From now on, we denote by $F$ the set of all dominant weights $\lambda$ of ${\bf GL}(n)$ such that 
$H^0_{B^-}(\lambda)=L(\lambda)$.
We denote by $F_d$ the subset of $F$ consisting of weights $\lambda$ of defect $d$.

\section{A reduction to the weights $\omega^a_{-i}$}

To a dominant weight $\lambda$ we assign the vector $m(\lambda)=(\lambda_1-\lambda_n, \ldots, \lambda_1-\lambda_i, \lambda_1-\lambda_{i-1},  \ldots, \lambda_1-\lambda_2)$. We impose the lexicographic order $\prec$ on the vectors $m(\lambda)$.

\begin{lm} \label{lm3.1}
If $\mu\unlhd \lambda$, then $m(\mu)\prec m(\lambda)$. Additionally, if $\mu\lhd \lambda$, then $m(\mu)\precneqq m(\lambda)$.
\end{lm}
\begin{proof} The condition 
$\mu \unlhd \lambda$ is equivalent to the series of inequalities 
\begin{equation}\label{eq1} \mu_1+\ldots + \mu_i\leq \lambda_1+\ldots +\lambda_i\end{equation}
for $i=1, \dots, n$ and an equality
$\mu_1+\ldots +\mu_n=\lambda_1+\ldots + \lambda_n$.

We assume the equality $\mu_1+\ldots +\mu_n=\lambda_1+\ldots + \lambda_n$. Then 
the inequalities $(\ref{eq1})$ imply $2\mu_1+\mu_2+\ldots + \mu_{n-1}\leq 2 \lambda_1+\lambda_2+\ldots + \lambda_{n-1}$, which is equivalent to 
$\mu_1-\mu_n\leq \lambda_1-\lambda_n$.
If $\mu_1-\mu_n=\lambda_1-\lambda_n$, then $\mu_1=\lambda_1$, $\mu_1+\ldots +\mu_{n-1}=\lambda_1+\ldots + \lambda_{n-1}$ and $\mu_n=\lambda_n$. The inequalities $(\ref{eq1})$ imply $2\mu_1+\mu_2+\ldots + \mu_{n-2}\leq 2 \lambda_1+\lambda_2+ \ldots + \lambda_{n-2}$. Therefore, $\mu_1-\mu_{n-1}\leq \lambda_1-\lambda_{n-1}$.
If $\mu_1-\mu_{n-1}=\lambda_1-\lambda_{n-1}$, then  $\mu_1+\mu_2+\ldots + \mu_{n-2}= \lambda_1+\lambda_2+ \ldots + \lambda_{n-2}$ and $\mu_{n-1}=\lambda_{n-1}$.
Proceeding like this we establish that $\mu_1-\mu_i\leq \lambda_1-\lambda_i$ for each $i=n, \ldots, 2$.
It is clear that $m(\mu)= m(\lambda)$ and $\mu_1+\ldots +\mu_n=\lambda_1+\ldots + \lambda_n$ implies 
$\mu=\lambda$.
\end{proof}

\begin{lm}\label{lm3.2}
Assume $\mu\in F$, $i>1$ and a weight $\kappa=\mu+2\epsilon_i$ or a weight $\kappa=\mu+\epsilon_i+\epsilon_{i+1}$ provided $\mu_i=\mu_{i+1}$ is dominant. Then $\mu\sim\kappa$ and 
$m(\kappa)\precneqq m(\mu)$.
\end{lm}
\begin{proof}
If $\kappa$ as above is dominant, then $\mu\sim \kappa$ by Proposition \ref{main}. 
Since $i>1$, we have $\kappa_1=\lambda_1$. We have $\kappa_j=\lambda_j$ for $j<i$ 
and $\kappa_i>\lambda_i$, showing that $m(\kappa)\precneqq m(\lambda)$.
\end{proof}

For any integer $a$ and $i=0, \ldots n$ define $\omega^a_{-i}=(a, \ldots, a, a-1, a-2, \ldots, a-i)$.
Note $\omega^a_{-n}=(a-1, a-2, \ldots, a-n)$.

\begin{pr}\label{p3.3}
Every dominant weight $\lambda$ of ${\bf P} (n)$ is linked to one of the weights $\omega^a_{-i}$ for some $a$ and $i=0, \ldots, n$ such that $\omega^a_{-i}\in F_0$.
\end{pr}
\begin{proof}
Let $\sigma$ be a dominant weight linked to $\lambda$ such that $m(\sigma)$ is minimal. Then by Lemma \ref{even linkage}, $\sigma\in F_d$ for some $d\geq 0$. By Lemma \ref{lm3.2}, there is no index $i>1$ such that $\sigma+\alpha_i$, where $\alpha_i=2\epsilon_i$ or $\epsilon_i+\epsilon_{i+1}$ provided $\sigma_i=\sigma_{i+1}$, 
is dominant. 
Therefore, we can assume that each $\sigma_i-\sigma_{i+1}$ is either 0 or 1. Additionally, if 
$\sigma_j-\sigma_{j+1}=1$ and $\sigma_{j+1}-\sigma_{j+2}=0$, then  
$\sigma+\epsilon_{j+1}+\epsilon_{j+2}\precneqq \sigma$. This implies that every $\sigma_k-\sigma_{k+1}=0$ must appear before every $\sigma_j-\sigma_{j+1}=1$. Therefore, $\sigma=\omega^a_{-i}$ for some $a$ and $i=0, \ldots, n$.
\end{proof}

There is a further reduction for a small value of $p$. 

\begin{pr}\label{p3.4}
Assume $p\leq 2n-1$. Then every dominant weight of ${\bf P} (n)$ is linked to one of $\omega^a_{-i}$ where $0\leq i\leq \frac{p-1}{2}$.
\end{pr}
\begin{proof}
Let $\sigma$ be a dominant weight with minimal $m(\sigma)$ linked to a given dominant weight $\lambda$. Proposition \ref{p3.3} implies that $\sigma=\omega_{-i}^a$ for some $a$ and $1\leq i\leq n$.	
The condition $p\leq 2n-1$ is equivalent to $j=\frac{p+1}{2}\leq n$.
We will show that if $j\leq i\leq n$, then $m(\omega^a_{-i})$ is  not minimal.

Consider $j\leq i\leq n$. Then 
\[\omega^a_{-i}\sim_{ev}a^{n-i}(a-1)\ldots(a-i+j+1)(a-i+j-1)^2(a-i+j-2)\ldots (a-i+1)^2=\mu\]
because $\mu=s_{\epsilon_{n-j}-\epsilon_n, p}\bullet \omega^a_{-i}$.
Since $\mu$ is dominant and $\mu\lhd \omega^a_{-i}$, we conclude that $m(\omega^a_{-i})$ is not minimal by Lemma \ref{lm3.1}.
\end{proof}

\section{Jantzen's criteria of irreducibility of $V_{ev}(\lambda)$}

Let $V_{ev}(\lambda)$ be the Weyl module for ${\bf GL} (n)$. According to II, 8.21 of \cite{jan} or Satz 9 of \cite{jan-1}, we have the following: 

\begin{pr}\label{p4.1}
Let $\lambda$ be a dominant weight. Then the Weyl module $V_{ev}(\lambda)=H^0_{ev}(-w_0(\lambda))^*$ is irreducible if and only if the following condition $(*_{i,j})$ is satisfied for every positive root $\alpha=\epsilon_i-\epsilon_j$, where $1\leq i<j\leq n$:
\begin{equation}\tag{$*_{i,j}$}
\begin{aligned}&
\text{Write } \lambda_i-\lambda_j+(j-i)=cp^s+bp^{s+1}, \text{ where } 0< c<p. \text{ Then there are 
integers } \\
& i=i_0<i_1<\ldots < i_b<i_{b+1}=j \text{ such that } \lambda_{i_r}-\lambda_{i_{r+1}}+(i_{r+1}-i_r)=p^{s+1} \text{ for each } \\
&r=0, \ldots b-1, \text{ or for each } r=1, \ldots, b.
\end{aligned}
\end{equation}
\end{pr}

Since the condition $(*_{ij})$ is trivially satisfied if $b=0$, in later considerations we will assume that $b>0$. 

Since $H^0_{ev}(\lambda)$ and $V_{ev}(\lambda)$ have the same formal characters, their images in the Grothendieck group of $GL(n)$ coincide. Therefore $V_{ev}(\lambda)$ is irreducible if and only if $H^0_{ev}(\lambda)$ is irreducible.
Thus, Proposition \ref{p4.1} gives us effective criteria to describe when $\lambda\in F_0$.
Every time we claim that a given weight $\lambda\in F_0$, it has been verified using this criterion.

To simplify the notation, we denote $d_{k,l}(\lambda)=\lambda_k-\lambda_l+(l-k)$.
The expressions $d_{k,l}(\lambda)$ are addititive, meaning that 
$d_{k,l}(\lambda)+d_{l,t}(\lambda)=d_{k,t}(\lambda)$ for $k\leq l\leq t$. Since $d_{k,l}(\lambda)\geq 0$, we have 
$d_{k,l}(\lambda)\leq d_{u,v}(\lambda)$ whenever $u\leq k\leq l\leq v$. 

The following lemma will simplify the verification that certain weights $\lambda$ belong to $F_0$. 

\begin{lm}\label{new2}
Assume $\lambda$ is a dominant weight and there are indices $1\leq k\leq l\leq n$ such that $d_{1,k}(\lambda)<p$, 
$\lambda_{k}=\ldots =\lambda_l$, and $d_{l,n}(\lambda)<p$.

Assume $1\leq u<v\leq n$ and write $d_{u,v}(\lambda)=cp^s+bp^{s+1}$, where $0<c<p$.

Then for each $0< j < b$,
there are indices $u<r_j,t_j<v$ such that $d_{u,r_j}(\lambda)=cp^s+jp^{s+1}$ and 
$d_{u,t_j}(\lambda)=jp^{s+1}$.

If there is an index $r$ such that $d_{u,r}(\lambda)=cp^s$ or there is an index $t$ such that 
$d_{u,t}(\lambda)=bp^{s+1}$, then the condition $(*_{u,v})$ is satisfied. 
\end{lm}
\begin{proof}
For $j=1, \ldots, b-1$ define $r_j$ to be the smallest index such that 
$d_{u,r_j}(\lambda)\geq cp^s+jp^{s+1}$. 
We claim that $k< r_1 < \ldots < r_{b-1} \leq l$, which implies $d_{u, r_j}(\lambda)=c p^s+j p^{s+1}$.
Since $d_{u, r_j}(\lambda)> p$, the assumption $d_{1, k}(\lambda)< p$ implies $k< r_j$. On the other hand, if
$v\leq l$, then $r_j\leq l$ is obvious. If $v>l$, then 
$d_{l, n}(\lambda)< p$ implies $d_{u, l}(\lambda)=d_{u, v}(\lambda)-d_{l, v}(\lambda)> cp^s+(b-1)p^{s+1}\geq cp^s+jp^{s+1}$, showing that $r_j\leq l$.
Finally, if $k< r_j\leq l$, then $d_{u, r_j-1}(\lambda)=d_{u, r_j}(\lambda)-1$, which implies $d_{u, r_j}(\lambda)=c p^s+j p^{s+1}$.

For $j=1, \ldots, b-1$ define $t_j$ to be the smallest index such that 
$d_{u,t_j}(\lambda)\geq jp^{s+1}$. 
We claim that $k< t_1 < \ldots < t_{b-1} \leq l$, which implies $d_{u, t_j}(\lambda)=j p^{s+1}$.
Since $d_{u, t_j}(\lambda)> p$, the assumption $d_{1, k}(\lambda)< p$ implies $k< t_j$. On the other hand, if
$v\leq l$, then $t_j\leq l$ is obvious. If $v>l$, then 
$d_{l, n}(\lambda)< p$ implies $d_{u, l}(\lambda)=d_{u, v}(\lambda)-d_{l, v}(\lambda)> cp^s+(b-1)p^{s+1}\geq jp^{s+1}$, showing that $t_j\leq l$.
Finally, if $k< t_j\leq l$, then $d_{u, t_j-1}(\lambda)=d_{u, t_j}(\lambda)-1$, which implies $d_{u, t_j}(\lambda)=j p^{s+1}$.

Assume there is an index $t$ such that $d_{u,t}(\lambda)=bp^{s+1}$ and set $u=t_0$, $t=t_b$. Then $d_{u,t_j}(\lambda)=jp^{s+1}$ for $j=1, \ldots, b$ imply
$d_{t_j,t_{j+1}}(\lambda)=p^{s+1}$ for $j=0, \ldots b-1$, showing that $(*_{u,v})$ is satisfied.

Assume there is an index $r$ such that $d_{u,r}(\lambda)=cp^s$ and set $r=r_0$, $v=r_b$. Then
$d_{u,r_b}(\lambda)=cp^s+bp^{s+1}$, 
$d_{u,r_0}(\lambda)=cp^s$ and $d_{u,r_j}(\lambda)=cp^s+jp^{s+1}$ for $j=1, \ldots, b-1$ imply 
$d_{r_j,r_{j+1}}(\lambda)=p^{s+1}$ for $j=0,\ldots, b-1$, which implies that the condition $(*_{u,v})$ is satisfied.
\end{proof}

Assume $1\leq u<v\leq n$ and $d_{u,v}(\lambda)=cp^s+bp^{s+1}$, where $0<c<p$. 
We will denote by $r_0$ the smallest index such that 
$d_{u,r_0}(\lambda)\geq cp^s$ and by $t_b$ the smallest index such that 
$d_{u,t_b}(\lambda)\geq bp^{s+1}$.

\begin{lm}\label{new}
Assume $\lambda$ is a dominant weight and there are indices $1\leq k\leq l\leq n$ such that $d_{1,k}(\lambda)<p$, 
$\lambda_{k}=\ldots =\lambda_l$, and $d_{l,n}(\lambda)<p$.

Assume that for $1\leq u<v\leq n$ such that $d_{u,v}(\lambda)=cp^s+bp^{s+1}$, where $0<c<p$, 
one of the following conditions is satisfied.
\begin{itemize}
\item{$s>0$}
\item{there is an index $t>l$ such that $d_{u,t}(\lambda)\equiv 0 \pmod p$.}
\end{itemize}
Then the condition $(*_{u,v})$ is satisfied.
If one of the above conditions is satisfied for every $1\leq u<v\leq n$,
then $\lambda\in F_0$.
\end{lm}
\begin{proof}
The assumption $d_{1,k}(\lambda)<p$ implies $\lambda$ has defect zero.
If $b=0$, then the condition $(*_{u,v})$ is automatically satisfied. Therefore we can assume $b>0$.

Assume $s>0$. Then $d_{u, r_0}(\lambda)> p$ and the assumption $d_{1, k}(\lambda)< p$ implies $k< r_0$. 
On the other hand, if
$v\leq l$, then $r_0\leq l$ is obvious. If $v>l$, then $d_{l, n}(\lambda)< p$ implies $d_{u, l}(\lambda)=d_{u, v}(\lambda)-d_{l, v}(\lambda)> cp^s+(b-1)p^{s+1}\geq cp^s$, hence $r_0\leq l$.
As before, $k< r_0\leq l$ implies $d_{u, r_0-1}(\lambda)=d_{u, r_0}(\lambda)-1$ showing that $d_{u, r_0}(\lambda)=c p^s$. By Lemma \ref{new2}, the condition $(*_{u,v})$ is satisfied.

If  $d_{u,t}(\lambda)\equiv 0 \pmod p$ for $t>l$, then $d_{l,n}(\lambda)<p$ implies $s=0$, $d_{u,t}(\lambda)=bp$. By Lemma \ref{new2},
the condition $(*_{u,v})$ is satisfied.

The final statement follows from Proposition \ref{p4.1}.
\end{proof}

We usually apply the previous lemmas to the case when $0\leq i\leq \frac{p-1}{2}$, $l=n-i$ and $\lambda_n\geq a-i$.
These assumptions imply $d_{l,n}(\lambda)<p$.

Next, we investigate specific dominant weights $\lambda$ of defect zero.

\begin{lm}\label{l4.3}
Let $\lambda$ be a dominant weight such that
$\lambda_1=\ldots =\lambda_{n-i}=a$ 
and $\lambda_n \geq a-i$, where $0\leq i\leq \frac{p-1}{2}$. Then 
$\lambda\in F_0$.
\end{lm}
\begin{proof}
Consider 
$1\leq u<v \leq n$ and $d_{u,v}(\lambda)=cp^s+bp^{s+1}$, where $0<c<p$ and $b>0$.

Then $d_{n-i,n}(\lambda)=\lambda_{n-i}-\lambda_n +(n-(n-i)) <p$ implies $r_0\leq n-i$.
As before, $d_{u,r_0-1}(\lambda)=d_{u,r_0}(\lambda)-1$ implies $d_{u,r_0}=cp^s$, and the condition $(*_{u,v})$ is satisfied. Since the defect of $\lambda$ is zero, we get $\lambda\in F_0$.
\end{proof}

\begin{lm}\label{l4.4}
Assume  $\lambda$ is dominant of defect zero such that $\lambda_k=\ldots=\lambda_{n-i}=a$ for some $1\leq k\leq n-i$ and $\lambda_n\geq a-i$, where $0\leq i\leq \frac{p-1}{2}$.

If all conditions $(*_{u,v})$ where $1\leq u< k,v$ are satisfied, then $\lambda\in F_0$.

In particular, assume $p\geq n$ and $\lambda_1\leq a+3$.

If for $1\leq u\leq k-1$ there is an index $i_u$ such that $d_{u,i_u}(\lambda)=p$, then each condition 
$(*_{u,v})$ is satisfied.

If for $1\leq u\leq k-1$ and each $v$ such that $d_{u,v}(\lambda)>p$ there is an index $j$ such that $d_{j,v}(\lambda)=p$, then all conditions $(*_{u,v})$ are satisfied. 
\end{lm}
\begin{proof}
If  $u\geq k$, then the condition $(*_{u,v})$ is satisfied by Lemma \ref{l4.3}.

If $\lambda_1\leq a+3$, then $d_{1n}(\lambda)\leq i+3+(n-1)\leq \frac{p-1}{2}+p+2\leq 2p$, showing that $s=0$ and 
$b\leq 1$ for each $1\leq u<v\leq n$.

If $d_{u,i_u}(\lambda)=p$, then Lemma \ref{new2} shows that $(*_{u,v})$ is satisfied for $u<k$. 

If $d_{u,v}(\lambda)\leq p$, then $b=0$ and the condition $(*_{u,v})$ is trivial.
If $u<k$ and $d_{u,v}(\lambda)>p$, then $d_{j,v}(\lambda)=p$, which implies 
$d_{u,j}(\lambda)=c$ and by Lemma \ref{new2}, the condition $(*_{u,v})$ is satisfied.
\end{proof}

\begin{lm}\label{l4.5}
Assume  $\lambda$ is dominant of defect zero such that
$\lambda_{k}=\ldots=\lambda_l=a$ for some $1\leq k\leq l$, $l\geq n-i$ and $\lambda_n\geq a-i$, where $0\leq i\leq \frac{p-1}{2}$.
Assume $d_{1,k}(\lambda)<p$ and for each $u=1,\ldots, k-1$ there is an integer $r_u$ such 
that $r_up\leq d_{u,l}(\lambda)\leq d_{u,n}(\lambda)\leq (r_u+1)p$. Then $\lambda\in F_0$.
\end{lm}
\begin{proof}
From the assumptions on $\lambda$, we have $d_{l,n}(\lambda)\leq 2i<p$.
If  $1\leq u< k,v$, and $b>0$, then $k<t_b\leq l$. Indeed, the inequality $k<t_b$ follows from 
$d_{1,k}(\lambda)<p$. If $v<n$, then $d_{u,v}(\lambda)<(r_u+1)p$. If 
$s>0$, then $t_b\leq l$.  If $v=n$ and $s=0$, then $p$ does not divide $d_{u,n}(\lambda)$ which implies 
$d_{u,n}(\lambda)<(r_u+1)p$. The assumptions $r_up\leq d_{u,l}(\lambda)\leq d_{u,n}(\lambda)\leq (r_u+1)p$ imply that $t_b\leq l$. Since $k<t_b\leq l$, we infer $d_{u, t_b-1}(\lambda)=d_{u,t_b}(\lambda)-1$, which implies $d_{u,t_b}(\lambda)=bp^{s+1}$, and Lemma \ref{new2} shows that 
the condition $(*_{u,v})$ is satisfied if $u<k$. If $u\geq k$, then $(*_{u,v})$ is satisfied by Lemma \ref{l4.3}. Therefore by Lemma \ref{l4.4} we conclude $\lambda\in F_0$.
\end{proof}

\begin{lm}\label{l4.8}
Let $\lambda=(a+2)(a+1)^{1+2j}a^{n-i-2j}$. If $d_{1,2+2j}(\lambda)=2+2j<p$, then $\lambda\in F_0$.
\end{lm}
\begin{proof}
Assume $1\leq u<v\leq n$, $d_{u,v}(\lambda)=cp^s+bp^{s+1}$ with $0<c<p$ and $b>0$. 
For $a=1, \ldots, b$ denote by $t_a$ the smallest index such that $d_{u,t_a}(\lambda)\geq ap^{s+1}$.
The inequality $d_{1,2+2j}(\lambda)<p$ implies that each $t_a> 2+2j$. Since $d_{u,t_a-1}(\lambda)=d_{1,t_a}(\lambda)-1$, we infer that $d_{u,t_a}(\lambda)=ap^{s+1}$ for each $a=1, \ldots, b$, hence the condition $(*_{u,v})$ is satisfied. Therefore, $\lambda\in F_0$.
\end{proof}

\begin{lm}\label{l4.6}
Assume  $\lambda$ is a dominant weight of defect zero and $0\leq i\leq \frac{p-1}{2}$.

If $\lambda_1=a+1$, $\lambda_2=\ldots=\lambda_{n-i}=a$, $\lambda_n\geq a-i$
and 
$d_{1,v}(\lambda)\not\equiv 1 \pmod p$ for every $n-i<v$, then $\lambda\in F_0$.

If $\lambda_1=(a+2)$, $\lambda_2=(a+1)$, $\lambda_3=\ldots =\lambda_{n-i}=a$, $\lambda_n\geq a-i$, $p>3$,
$d_{1,v}(\lambda)\not\equiv 1 \pmod p$ for every $n-i<v$ and 
$d_{2,v}(\lambda)\not\equiv 1 \pmod p$ for every $n-i<v$, 
then $\lambda\in F_0$.
\end{lm}
\begin{proof}
Assume $\lambda_1=a+1$, $\lambda_2=\ldots=\lambda_{n-i}=a$, $\lambda_n\geq a-i$.
By Lemma \ref{l4.4}, it is enough to verify conditions $(*_{1,v})$ for every $1<v\leq n$.
If $v\leq n-i$, then the condition $(*_{1,v})$ is clearly satisfied.
Therefore we can assume that $n-i<v$.
Write $d_{1,v}(\lambda)=cp^s+bp^{s+1}$ with $0<c<p$. If $b=0$, the condition $(*_{1,v})$ is trivial, so we can assume $b>0$.
Since $d_{1,v}(\lambda)\not\equiv 1 \pmod p$, we have $cp^s>1$.
Then $d_{1,cp^s}(\lambda)=cp^s$ and Lemma \ref{new2} implies that the condition $(*_{1,v})$ is satisfied,  showing that $\lambda\in F_0$.

Assume $\lambda_1=(a+2)$, $\lambda_2=(a+1)$, $\lambda_3=\ldots =\lambda_{n-i}=a$, $\lambda_n\geq a-i$.
By Lemma \ref{l4.4}, it is enough to verify conditions $(*_{1,v})$ and $(*_{2,v})$ for every $v$.
Every condition $(*_{2,v})$ for $v\leq n-i$ is clearly satisfied. The condition $(*_{1,v})$ for $v\leq n-i$ is satisfied if $p>3$. Therefore we can assume $n-i<v$.

Write $d_{2,v}(\lambda)=c'p^{s'}+b'p^{s'+1}$ with $0<c'<p$. If $b'=0$, the condition $(*_{2,v})$ is trivial, so we can assume $b'>0$. Since $d_{2,v}(\lambda)\not \equiv 1 \pmod p$, we have $c'p^{s'}>1$. 
If $c'p^{s'}\geq 3$, then $d_{2, c'p^{s'}}(\lambda)=c'p^{s'}$,
and Lemma \ref{new2} implies that the condition $(*_{2,v})$ is satisfied.
If $c'p^{s'}=2$, then $d_{2, 3}(\lambda)=2$, and Lemma \ref{new2} implies that the condition $(*_{2,v})$ is satisfied.

Write  $d_{1,v}=cp^s+bp^{s+1}$ with $0<c<p$. If $b=0$, the condition $(*_{1,v})$ is trivial, so we can assume $b>0$.
Since $d_{1,v}(\lambda)\not \equiv 1 \pmod p$, we have $cp^s>1$.
Since $d_{2,v}(\lambda)\not \equiv 1 \pmod p$, we have $cp^s\neq 3$. 
If $cp^s\geq 4$, then $d_{1,cp^s-1}=cp^s$, and Lemma \ref{new2} implies that the condition $(*_{1,v})$ is satisfied.
If $cp^s=2$, then $d_{1,2}(\lambda)=2$, and Lemma \ref{new2} implies that the condition $(*_{1,v})$ is satisfied. Using Lemma \ref{l4.4}, we conclude that $\lambda\in F_0$.
\end{proof}

\begin{lm}\label{l4.7}
If  $\lambda_1=\lambda_2=(a+1)$, $\lambda_3=\ldots =\lambda_{n-i}=a$ and $\lambda_n\geq a-i$, $0\leq i\leq \frac{p-1}{2}$, and $d_{2,v}(\lambda)\not\equiv 1 \pmod p$ for every $n-i<v$, 
then $\lambda\in F_0$.
\end{lm}
\begin{proof}
By Lemma \ref{l4.4}, it is enough to verify conditions $(*_{1,v})$ and $(*_{2,v})$ for every $v$.
Write $d_{2,v}(\lambda)=c'p^{s'}+b'p^{s'+1}$ with $0<c'<p$. If $b'=0$, the condition $(*_{2,v})$ is trivial, so we can assume $b'>0$.
If $v\leq n-i$, then $d_{2, b'p^{s'+1}+1}=b'p^{s'+1}$ and Lemma \ref{new2} implies that the condition $(*_{2,v})$ is satisfied.

Therefore we can assume $n-i<v$.
Since $d_{2,v}(\lambda)\not \equiv 1 \pmod p$, we have $c'p^{s'}>1$. 
Then $d_{2, c'p^{s'}+1}(\lambda)=c'p^{s'}$ and Lemma \ref{new2} implies that the condition $(*_{2,v})$ is satisfied.

Write  $d_{1,v}=cp^s+bp^{s+1}$ with $0<c<p$. If $b=0$, the condition $(*_{1,v})$ is trivial, so we can assume $b>0$. If $v\leq n-i$, then  $d_{1,bp^{s+1}}(\lambda)=bp^{s+1}$ and Lemma \ref{new2} implies that the condition $(*_{2,v})$ is satisfied.

Since $d_{2,v}(\lambda)\not \equiv 1 \pmod p$, we have $cp^s\neq 2$. 
If $cp^s\geq 3$, then $d_{1,cp^s}=cp^s$ and Lemma \ref{new2} implies that the condition $(*_{2,v})$ is satisfied.

If $cp^s=1$, then $d_{1,2}(\lambda)=1$ and Lemma \ref{new2} implies that the condition $(*_{2,v})$ is satisfied.
Using Lemma \ref{l4.4}, we conclude that $\lambda\in F_0$.
\end{proof}

\section{Non-minimality of $\omega^a_{-i}$ for $2\leq i\leq \frac{p-1}{2}$ when $p\geq n$}

We say that a dominant weight $\lambda$ is minimal if there does not exist a dominant weight $\mu$
such that $\lambda\sim \mu$ and $\mu\precneqq \lambda$.

In this section,  denote $\lambda=\omega^a_{-i}$ for $2\leq i\leq \frac{p-1}{2}$ and assume $p\geq n$.

\begin{lm}\label{l5.1}
Assume $\mu$ is a dominant weight and $\mu_1-\mu_n\leq p-n+1$, or equivalently 
$d_{1n}(\mu)\leq p$. Then $\mu\in F_0$. 
\end{lm}
\begin{proof}
If $1\leq u<v\leq n$, then $d_{u,v}(\lambda)\leq p$. If $d_{u,v}(\lambda)<p$, then $b=0$ and the condition 
$(*_{u,v})$ is trivially satisfied. If $d_{u,v}(\lambda)=p$, then $u=1, v=n$ and condition $(*_{u,v})$ is again satisfied.
The statement is an immediate consequence of Proposition \ref{p4.1} 
\end{proof}

Note that $\mu_1-\mu_n\leq p-n+1$ if and only if the weight $\mu$ belongs to the lowest alcove $C_0$ or its upper wall $W_0$ under the dot action of the affine Weyl group.

First, we assume that $i-p+n\leq 1$. In this case, $d_{1n}(\lambda)=i+n-1\leq p$ and $\lambda\in F_0$.

\begin{pr}\label{p5.2}
If $2\leq i\leq p-n+1$, then $\omega^a_{-i}$ is not minimal. 
\end{pr}
\begin{proof}
We consider two cases depending on the parity of $n-i$.

Assume first that \underline{$n-i$ is even}. 
Since $\omega^a_{-n}=\omega^{a-1}_{-n+1}$, where $n-(n-1)=1$ is odd, we can assume $n-i\geq 2$.

If $n+i-1<p$, then we apply Proposition \ref{main} and Lemma \ref{l5.1} repeatedly to obtain 
\[\omega^a_{-i}=a^{n-i}(a-1)\ldots (a-i)\sim \mu^{(1)}=(a+1)^{n-i}(a-1)\ldots  (a-i),\]
where $\mu^{(1)}$ and all intermediate weights belong to $F_0$ because $n+i\leq p$.
Using Proposition \ref{main} and Lemma \ref{l5.1}, we add successive weights $2\epsilon_{n-i+k}$ for $k=1, \ldots, i$
and obtain that 
\[\mu^{(1)}\sim \mu^{(2)}=(a+1)^{n-i+1}a\ldots (a-i+2) \precneqq \omega^a_{-i}.\]

If $n+i-1=p$, then by Proposition \ref{main}  and Lemma \ref{l5.1} we have
\[\omega^a_{-i}=a^{n-i}(a-1)\ldots (a-i)\sim \mu^{(1)}=(a+1)^2a^{n-i-2}(a-1)\ldots  (a-i).\] 
Since $d_{1,n}(\mu^{(1)})=n+i=p+1$ and $d_{2,n}(\mu^{(1)})=p$, Lemma \ref{new2}  implies $\mu^{(1)}\in F_0$.
Using Proposition \ref{main} and Lemma \ref{l5.1}, we add successive weights $\epsilon_{2k-1}+\epsilon_{2k}$ for 
$k=2, \ldots, \frac{n-i}{2}$ and obtain
\[\mu^{(1)}\sim \mu^{(2)}=(a+1)^{n-i}(a-1)\ldots  (a-i).\] 
Using Proposition \ref{main} and Lemma \ref{l5.1}, we add successive weights $2\epsilon_{n-i+k}$ for $k=1, \ldots, i$
and obtain that 
\[\mu^{(2)}\sim \mu^{(3)}=(a+1)^{n-i+1}a\ldots (a-i+2) \precneqq \omega^a_{-i}.\]

Now assume that \underline{$n-i$ is odd}. Using Proposition \ref{main} and Lemma \ref{l5.1} repeatedly adding $2\epsilon_1$ we establish that
\[\omega^a_{-i} \sim \mu^{(1)}=(a+p-n-i+2) a^{n-i-1}(a-1)(a-2)\ldots (a-i).\]
Here all intermediary weights belong to $F_0$ but $\mu^{(1)}$ does not. Indeed, since
$d_{1,n}(\mu^{(1)})=p+1$, we have $s_{\epsilon_1-\epsilon_n,p}\bullet \mu^{(1)}=\mu^{(2)}$,
where \[\mu^{(2)}=(a+p-n-i+1)a^{n-i-1}(a-1)(a-2)\ldots (a-i+1)^2\] 
is dominant. Therefore, 
$\mu^{(1)}\sim_{ev} \mu^{(2)}$. 

Further, by Proposition \ref{main} and Lemma \ref{l5.1} we have
\[\mu^{(2)}\sim \mu^{(3)}=(a+p-n-i+1)a^{n-i-1}(a-1)(a-2)\ldots (a-i+2)^3.\]

On the other hand, starting with $\mu^{(4)}=(a+1)a^{n-i+1}(a-1)(a-2)\ldots (a-i+2)^3$, using Proposition \ref{main} and
Lemma \ref{l5.1}, we add successive weights $2\epsilon_1$ to obtain 
$\mu^{(4)}\sim \mu^{(3)}$. Since $\mu^{(4)}\precneqq \omega^a_{-i}$, the claim is proven.
\end{proof}

Therefore, we can assume $i-p+n=t>1$. We consider two cases when $t$ is even, and $t$ is odd.
The last inequality and $i\leq \frac{p-1}{2}$ imply $n-i >2$ if $t$ is even, and 
$n-i>3$ if $t$ is odd. 

\begin{lm}\label{l5.3}
If $t$ is even, then the weight $\omega^a_{-i}$ is not mininal with respect to $\prec$.
\end{lm}
\begin{proof}
Let us write $t=2s$. The assumptions $i-p+n=2s>0$ and $p\geq n$ imply $s<i$. 
By Lemma \ref{l4.3},  $\lambda=\omega^a_{-i}\in F_0$. Therefore 
\[\lambda\sim \lambda^{(1)}=(a+1)^2a^{n-i-2}(a-1)\ldots(a-i).\]
Since $d_{2,n-s+1}(\lambda^{(1)})=p+1$, we have $s_{\epsilon_2-\epsilon_{n-s+1},p}\bullet \lambda^{(1)}=\lambda^{(2)}$, where 
\[\lambda^{(2)}=(a+1)a^{n-i-1}(a-1)\ldots (a-i+s)^2(a-i+s-2)\ldots (a-i)\]
Thus, $\lambda^{(1)}\sim_{ev}\lambda^{(2)}$.

Since $d_{1,n-s}(\lambda^{(2)})=p$, by Lemma \ref{l4.4} we verify $\lambda^{(2)}\in F_0$.
If we change entries in $\lambda^{(2)}$ only in indices larger than $n-s$, then the resulting weight $\tilde{\lambda}$  satisfies $d_{1,n-s}(\tilde{\lambda})=p$. Therefore, Lemma \ref{l4.4} implies $\tilde{\lambda}\in F_0$.

Therefore
\[\begin{aligned}\lambda^{(2)}\sim& (a+1)a^{n-i-1}(a-1)\ldots (a-i+s)^3(a-i+s-3)\ldots (a-i)\\
\sim& (a+1)a^{n-i-1}(a-1)\ldots (a-i+s)^3(a-i+s-1)(a-i+s-4)\ldots (a-i)\\
\sim& (a+1)a^{n-i-1}(a-1)\ldots (a-i+s)^3(a-i+s-1)\\
&(a-i+s-2)(a-i+s-5)\ldots (a-i)\\
\ldots\\
\sim& (a+1)a^{n-i-1}(a-1)\ldots (a-i+s)^3(a-i+s-1)\ldots (a-i+2).
\end{aligned}\]
Since the last weight is smaller than $\omega^a_{-i}$, we conclude that $\omega^a_{-i}$ is not minimal.
\end{proof}

\begin{lm}\label{l5.4}
If $t$ is odd, then the weight $\omega^a_{-i}$ is not mininal with respect to $\prec$.
\end{lm}
\begin{proof}
Let us write $t=2s+1$. The assumptions $i-p+n=2s+1>0$ and $p\geq n$ imply $s<i$. 
By Lemma \ref{l4.3},  $\lambda=\omega^a_{-i}\in F_0$. Therefore 
\[\lambda\sim \lambda^{(1)}=(a+2)a^{n-i-1}(a-1)\ldots(a-i).\]
Since $d_{1,n-s-1}(\lambda^{(1)})=p$, Lemma \ref{l4.4} implies  $\lambda^{(1)}\in F_0$. Then 
\[\lambda^{(1)}\sim\lambda^{(2)}=(a+2)(a+1)^2a^{n-i-3}(a-1)\ldots (a-i).\]
Since $d_{3,n-s+1}(\lambda^{(2)})=p+1$, we have  $s_{\epsilon_3-\epsilon_{n-s+1},p}\bullet \lambda^{(2)}=\lambda^{(3)}$, 
where 
\[\lambda^{(3)}=(a+2)(a+1)a^{n-i-2}(a-1)\ldots (a-i+s)^2(a-i+s-2)\ldots (a-i).\]
Therefore, $\lambda^{(2)}\sim_{ev}\lambda^{(3)}$.
Since $d_{1,n-s-1}(\lambda^{(3)})=p$ and $d_{2,n-s}(\lambda^{(3)})=p$, by Lemma \ref{l4.4} we have $\lambda^{(3)}\in F_0$ and 
\[\lambda^{(3)}\sim\lambda^{(4)}=(a+2)(a+1)a^{n-i-2}(a-1)\ldots (a-i+s+1)^3(a-i+s-2)\ldots (a-i).\]
Since $d_{1,n-s-1}(\lambda^{(4)})=p$ and $d_{2,n-s+1}(\lambda^{(4)})=p$, Lemma \ref{l4.4} implies $\lambda^{(4)}\in F_0$ and
\[\begin{aligned}\lambda^{(4)}\sim\lambda^{(5)}=&(a+2)(a+1)a^{n-i-2}(a-1)\ldots (a-i+s+2)^3(a-i+s+1)\\
&(a-i+s-2)\ldots (a-i).\end{aligned}\]
Since $d_{1,n-s}(\lambda^{(5)})=p$ and $d_{2,n-s+1}(\lambda^{(5)})=p$, we obtain $\lambda^{(5)}\in F_0$.

If we change entries in $\lambda^{(5)}$ only in indices larger than $n-i$ but keep entries at indices $n-s$ and $n-s+1$ the same, then the resulting weight $\tilde{\lambda}$  satisfies $d_{1,n-s}(\tilde{\lambda})=p$ and $d_{2,n-s+1}(\tilde{\lambda})=p$, which implies $\tilde{\lambda}\in F_0$.
Therefore
\[\begin{aligned}\lambda^{(5)}\sim& (a+2)(a+1)a^{n-i-2}(a-1)\ldots (a-i+s+2)^3(a-i+s+1)(a-i+s)\\
&(a-s-3)\ldots (a-i)\\
\ldots\\
\sim& (a+2)(a+1)a^{n-i-2}(a-1)\ldots (a-i+s+2)^3(a-i+s+1)\ldots (a-i+2)\\
\sim& (a+2)(a+1)a^{n-i-2}(a-1)\ldots (a-i+s+3)^3(a-i+s+2)(a-i+s+1)\\
&\ldots (a-i+2)\\
\ldots\\
\sim& (a+2)(a+1)a^{n-i-2}(a-1)^3(a-2)\ldots  (a-i+2)\\
\sim& (a+2)(a+1)a^{n-i}(a-1)(a-2)\ldots  (a-i+2)=\lambda^{(6)}.\\
\end{aligned}\]

Since $\lambda^{(6)}\in F_0$, we have 
\[\lambda^{(6)}\sim\lambda^{(7)}=(a+2)(a+1)^3a^{n-i-2}(a-1)\ldots (a-i+2).\]
Since $d_{1,n-s}(\lambda^{(7)})=p$, $d_{2,n-s+1}(\lambda^{(7)})=p$, 
$d_{3,n-s+2}(\lambda^{(7)})=p+1$ and $d_{4,n-s+2}(\lambda^{(7)})$ $=p$, 
all conditions $(*_{1,l})$, $(*_{2,l})$ and $(*_{4,l})$ are satisfied by Lemma \ref{l4.4}.
Since $(*_{3,n-s+2})$ is also satisfied, consider $(*_{3,n-s+2+j})$, where $1\leq j\leq s-2$.
Then $d_{3,n-s+2+j}(\lambda^{(7)})=p+1+2j$ and $d_{3+2j,n-s+2+j}(\lambda^{(7)})=p$, where $5\leq 3+2j\leq n-i$.
By Lemma \ref{l4.4}, all conditions $(*_{3,l})$ are satisfed, which implies $\lambda^{(7)}\in F_0$.
Then
\[\lambda^{(7)}\sim \lambda^{(8)}=(a+2)(a+1)^5a^{n-i-4}(a-1)\ldots (a-i+2).\]
Since the equalities $d_{1,n-s}(\lambda^{(8)})=p$, $d_{2,n-s+1}(\lambda^{(8)})=p$, 
and $d_{4,n-s+2}(\lambda^{(8)})=p$ are carried over from previous equalities for 
$\lambda^{(7)}$, and $d_{6,n-s+3}(\lambda^{(8)})=p$, by Lemma \ref{l4.4}, all conditions $(*_{1,l})$, $(*_{2,l})$ $(*_{4,l})$ and $(*_{6,l})$ are satisfied.
The condition $(*_{3,n-s+2})$ and $(*_{3,n-s+2+j})$ for $2\leq j\leq s-2$ are also satisfied because
$d_{3+2j,n-s+2+j}(\lambda^{(8)})=p$, where $7\leq 3+2j\leq n-i$. The condition $(*_{3,n-s+3})$ is satisfied because
$d_{4,n-s+2}(\lambda^{(8)})=p$. Therefore, by Lemma \ref{l4.4}, all conditions $(*_{3,l})$ are satisfied.
The condition $(*_{5,n-s+3})$ is satisfied. Since
$d_{5,n-s+3+j}(\lambda^{(8)})=p+1+2j$ and $d_{5+2j,n-s+3+j}(\lambda^{(8)})=p$, where $7\leq 5+2j\leq n-i$,
by Lemma \ref{l4.4}, all conditions $(*_{5,l})$ are satisfied.
Therefore, $\lambda^{(8)}\in F_0$.

We continue like this through the sequence of weights
\[\mu^{(k)}=(a+2)(a+1)^{2k+1} a^{n-i-2k}(a-1)\ldots (a-i+2)\in F_0\]
for $k$ such that $2k\leq n-i$ until we reach 
\[\lambda^{(9)}=(a+2)(a+1)^{n-i+1} a^{0}(a-1)\ldots (a-i+2)\]
since $n-i$ is even.

We will verify that each $\mu^{(k)}\in F_0$.
First, 
\[d_{1,n-s}(\mu^{(k)})=d_{2,n-s+1}(\mu^{(k)})=\ldots =d_{2t,n-s+t}(\mu^{(k)})=\ldots =d_{2k,n-s+k}(\mu^{(k)})=p\]
and Lemma \ref{l4.4} implies that conditions $(*_{1,l})$ and $(*_{2t,l})$ for $1\leq t\leq s$ for $\mu^{(k)}$  are satisfied.
Consider $1\leq t\leq k-1$. If $t\geq s$, then $d_{2t+1,n}(\mu^{(k)})<p$ and all conditions $(*_{2t+1,l})$ are satisfied. If $t<s$, then $d_{2t+1,n-s+t+1+j}(\mu^{(k)})=p+1+2j$
for $0\leq j\leq s-t-1$. Since 
 $d_{2t+2+2j, n-s+t+1+j}(\mu^{(k)})=p$ if $j\leq k-t$ and $d_{2t+1+2j, n-s+t+1+j}(\mu^{(k)})=p$ if $j>k-t$, 
by Lemma \ref{l4.4}, all conditions  $(*_{2t+1,l})$ for $1\leq t\leq k-1$ are satisfied and $\mu^{(k)}\in F_0$.

Since all weights $\mu^{(k)}$ are odd-linked to $\lambda^{(9)}$ and 
$\lambda^{(9)}\precneqq \omega^a_{-i}$, we conclude that $\omega^a_{-i}$ is not minimal.
\end{proof}

Combining Proposition \ref{p5.2} and Lemmas \ref{l5.3}, \ref{l5.4}, we have proved the following statement.

\begin{pr}\label{p5.5}
If $p\geq n$ and $2\leq i$, then $\omega^a_{-i}$ is not minimal. 
\end{pr}

\section{Non-minimality of $\omega^a_{-i}$ for $2\leq i\leq \frac{p-1}{2}$ when $p\leq n$}

Assume $2\leq i\leq \frac{p-1}{2}$, which implies $p\geq 5$. Also assume $p\leq n$ and 
write $n=p\underline{n}+\overline{n}$, where $0\leq \overline{n}<p$.

Denote $\nu=1^i$. For a weight $\lambda$ and $1\leq j\leq n-i$ we define 
\[d_j(\lambda)=(d_{j,n-i+1}(\lambda), \ldots d_{j,n}(\lambda))\] 
and $\overline{d}_j(\lambda)=d_j(\lambda)-\underline{n}p\nu$. 

\begin{pr}\label{p6.1}
The weight $\omega^a_{-i}$ for $i\geq 2$ is not minimal if 
$i+1\leq \overline{n}\leq p-i+1$.
\end{pr}
\begin{proof}
The assumptions $i\leq \frac{p-1}{2}$ and $p\leq n$ imply $n\geq i+3$.
The last inequality guarantees that all powers of $a$ appearing below have nonnegative exponents.
Set $\lambda=\omega^a_{-i}=a^{n-i}(a-1)\ldots (a-i)$. Then $\lambda\in F_0$ by Lemma \ref{l4.3}.

Assume \underline{$\overline{n}-i$ is odd}. 

Define 
\[\lambda^{(j)}=(a+2j)a^{n-i-1}(a-1)\ldots (a-i)\]
 for $j=1, \ldots, \frac{p-i-\overline{n}+2}{2}=r$. 
Then  for $j<r$ there is 
$\overline{d}_{1,n-i}(\lambda^{(j)})=-i-1+\overline{n}+2j\geq 0$ and 
$\overline{d}_{1,n}(\lambda^{(j)})=i-1+\overline{n}+2j\leq i-1+\overline{n}+p-i-\overline{n}<p$. Therefore by Lemma \ref{l4.5}, we have  $\lambda^{(j)}\in F_0$ for $1\leq j<r$.

Thus $\lambda\sim \lambda^{(r)}$ and 
$d_{1n}(\lambda^{(r)})=(\underline{n}+1)p+1$, which implies 
$s_{\epsilon_1-\epsilon_{n},(\underline{n}+1)p}\bullet \lambda^{(r)}=\lambda^{(r+1)}$, 
where
\[\lambda^{(r+1)}=(a+2r-1)a^{n-i-1}(a-1)\ldots (a-i+2)(a-i+1)^2,\]
hence $\lambda^{(r)}\sim_{ev} \lambda^{(r+1)}$.
Since $\overline{d}_{1,n-i}(\lambda^{(r+1)})=p-2i>0$ and 
$\overline{d}_{1,n}(\lambda^{(r+1)})=p-1$, Lemma \ref{l4.5} implies $\lambda^{(r+1)}\in F_0$ and 
\[\lambda^{(r+1)}\sim \lambda^{(r+2)}=(a+2r-1)a^{n-i-1}(a-1)\ldots (a-i+2)^3.\]

Define \[\mu^{(t)}=(a-1+2t)a^{n-i-1}(a-1)\ldots (a-i+2)^3\] for $t=1, \ldots, r$. Since
$\overline{d}_{1,n-i}(\mu^{(t)})=\overline{n}-i+2t-2\geq 1$
and $\overline{d}_{1,n}(\mu^{(t)})=p-2+2t-2r\leq p-2$,
Lemma \ref{l4.5} implies that all $\mu^{(t)}\in F_0$. Since $\mu^{(r)}=\lambda^{(r+2)}$, 
$\mu^{(1)}\sim \lambda^{(r+2)}\sim \lambda$. However, $\mu^{(1)}=(a+1)a^{n-i-1}(a-1)\ldots (a-i+2)^3\precneqq \lambda$ shows $\lambda$ is not minimal.

Assume \underline{$\overline{n}-i$ is even}. 
Define \[\lambda^{(j)}=(a+2j+1)(a+1)a^{n-i-2}(a-1)\ldots (a-i)\] for $j=0, \ldots, r=\frac{p-i-\overline{n}+1}{2}$. 
Since $\lambda\in F_0$, we have $\lambda\sim \lambda^{(0)}$.

Then for $0\leq j<r$ we have 
\[\overline{d}_{1,n-i}(\lambda^{(j)})=-i+\overline{n}+2j\geq 1, \]
\[\overline{d}_{1,n}(\lambda^{(j)})=i+\overline{n}+2j\leq i+\overline{n}+p-i-\overline{n}-1=p-1,\] 
\[\overline{d}_{2,n-i}(\lambda^{(j)})=-i-1+\overline{n}\geq 0, \]
\[\overline{d}_{2,n}(\lambda^{(j)})=i-1+\overline{n}\leq p,\] 
and by Lemma \ref{l4.5}, we conclude $\lambda^{(j)}\in F_0$ which implies $\lambda\sim \lambda^{(r)}$.

Since
$d_{1n}(\lambda^{(r)})=(\underline{n}+1)p+1$, we have 
$s_{\epsilon_1-\epsilon_{n},(\underline{n}+1)p}\bullet \lambda^{(r)}=\lambda^{(r+1)}$, where 
\[\lambda^{(r+1)}=(a+2r)(a+1)a^{n-i-2}(a-1)\ldots (a-i+2)(a-i+1)^2.\]
Therefore $\lambda^{(r)}\sim_{ev}\lambda^{(r+1)}$.
Since 
\[\overline{d}_{1,n-i}(\lambda^{(r+1)})=p-2i>0, \quad \overline{d}_{1,n}(\lambda^{(r+1)})=p-1,\]
\[\overline{d}_{2,n-i}(\lambda^{(r+1)})=1+\overline{n}-i-2\geq 0 \text{ and }
\overline{d}_{2,n}(\lambda^{(r+1)})=i+\overline{n}-2\leq p-1,\]
Lemma \ref{l4.5} implies $\lambda^{(r+1)}\in F_0$ and 
\[\lambda^{(r+1)}\sim \mu=(a+2r)(a+1)a^{n-i-2}(a-1)\ldots (a-i+2)^3.\]

Set $\mu=\mu^{(2)}$ and for $j=3, \ldots i-1$ define  
\[\mu^{(j)}=(a+2r)(a+1)a^{n-i-2}\ldots (a-i+j)^3(a-i+j-1)\ldots (a-i+2).\]
Since 
\[\overline{d}_{1,n-i}(\mu^{(j)})=p-2i>0, \quad \overline{d}_{1,n}(\mu^{(j)})=p-2,\]
\[\overline{d}_{2,n-i}(\mu^{(j)})=1+\overline{n}-i-2\geq 0 \text{ and } 
\overline{d}_{2,n}(\mu^{(j)})=i+\overline{n}-3\leq p-2,\] 
Lemma \ref{l4.5} implies $\mu^{(j)}\in F_0$  for all $j=2, \ldots, i-1$. Therefore, 
\[\begin{aligned}\mu\sim &\mu^{(i-1)}=
(a+2r)(a+1)a^{n-i-2}(a-1)^3(a-2)\ldots (a-i+2)\\
\sim &\mu^{(i)}=
(a+2r)(a+1)a^{n-i}(a-1)(a-2)\ldots (a-i+2).\end{aligned}\]

For $t=1, \ldots, s$ define 
\[\kappa^{(t)}=(a+2t)(a+1)a^{n-i}(a-1)(a-2)\ldots (a-i+2).\] Since
\[\overline{d}_{1,n-i}(\kappa^{(t)})=2t+\overline{n}-i-1>0, \quad 
\overline{d}_{1,n}(\kappa^{(t)})=2t+i-2+\overline{n}-1\leq p-2,\] 
\[\overline{d}_{2,n-i}(\kappa^{(t)})=1+\overline{n}-i-2\geq 0
\text{ and }
\overline{d}_{2,n}(\kappa^{(t)})=i+\overline{n}-3\leq p-2,\] 
Lemma \ref{l4.5} implies 
$\kappa^{(t)}\in F_0$ and 
\[\mu^{(i)}\sim \kappa^{(1)}=(a+2)(a+1)a^{n-i}(a-1)(a-2)\ldots (a-i+2).\]

For $j=0, \ldots, \frac{\overline{n}-i+2}{2}=r$ define 
\[\iota^{(j)}=(a+2)(a+1)^{1+2j}a^{n-i-2j}(a-1)(a-2)\ldots (a-i+2).\]

Consider $j<r$.
If $1<t\leq 2j+2$, then 
\[\overline{d}_{t,n-i+2}(\iota^{(j)})=1+\overline{n}-i+2-t\geq 1+\overline{n}-i+2-2j-2\geq 1+\overline{n}-i-\overline{n}+i>0\]
and
\[\overline{d}_{t,n}(\iota^{(j)})=i+\overline{n}-t-1\leq p-t<p.\]
Then by Lemma \ref{l4.5}, we infer $\iota^{(j)}\in F_0$ for each $j<r$. 

If $\underline{n}=0$, then 
\[\kappa^{(1)}\sim \iota^{(r-1)}=(a+2)(a+1)^{n-i+1}(a-1)\ldots (a-i+2)\prec \lambda.\]

If $\underline{n}>0$  and $i\geq 3$, then 
\[\kappa^{(1)}\sim \iota^{(r)}=(a+2)(a+1)^{\overline{n}-i+3}a^{\underline{n}p-2}(a-1)\ldots(a-i+2).\]

Since $d_{\overline{n}-i+4, n-i+3}(\iota^{(r)})=\underline{n}p+1$, we obtain 
\[s_{\epsilon_{\overline{n}-i+4}-\epsilon_{n-i+3},\underline{n}p}\bullet \iota^{(r)}=\iota^{(r+1)},\]
where 
\[\iota^{(r+1)}= (a+2)(a+1)^{\overline{n}-i+2}a^{\underline{n}p}(a-2)\ldots (a-i+2).\]
Therefore \[\iota^{(r)}\sim_{ev}\iota^{(r+1)}\precneqq \lambda.\]

Now assume $\underline{n}>0$ and $i=2$. In this case, we use
$\lambda\sim \kappa^{(1)}=(a+2)(a+1)a^{n-2}$. Since $n\geq p$, we can use Lemma \ref{l4.8} repeatedly
to get 
\[\kappa^{(1)}\sim (a+2)(a+1)^{p-2}a^{n-p+1}=\nu^{(1)}.\] 
Since 
$d_{1,p}(\nu^{(1)})=p+1$, we obtain $s_{\epsilon_1-\epsilon_{p},p}\bullet \nu^{(1)}=\nu^{(2)}$,
where $\nu^{(2)}=(a+1)^{p}a^{n-p}$.
Therefore, $\nu^{(1)}\sim_{ev}\nu^{(2)}\precneqq \lambda$.
\end{proof}

\begin{pr}\label{p6.2}
The weight $\omega^a_{-i}$ for $i\geq 2$ is not minimal if 
$\overline{n}\geq p-i+2$.
\end{pr}
\begin{proof}
The assumption $\overline{n}\geq p-i+2$ implies $\overline{n}\geq i+3$, which implies that all powers of $a$ appearing below have nonnegative exponents.
Write $\lambda=\omega^a_{-i}=a^{n-i}(a-1)\ldots (a-i)$ and $\overline{n}=p-t$.

Assume \underline{$\overline{n}-i$ is odd}.

Since $\lambda\in F_0$ by Lemma \ref{l4.3}, we have 
\[\lambda\sim (a+2)a^{n-i-1}(a-1)\ldots (a-i)=\lambda^{(1)}.\]

If $r=\frac{i-t}{2}$, then 
\[d_{1,n-r}(\lambda^{(1)})=2+i-r+n-r-1=1+i+n-i+t=(\underline{n}+1)p+1.\] 
Therefore, $s_{\epsilon_1-\epsilon_{n-r},(\underline{n}+1)p}\bullet \lambda^{(1)}=\lambda^{(2)}$, where
\[\lambda^{(2)}=(a+1)a^{n-i-1}(a-1)\ldots (a-i+r+2)(a-i+r+1)^2(a-i+s-1)\ldots (a-i),\]
which implies $\lambda^{(1)}\sim_{ev}\lambda^{(2)}$.

Since 
\[\overline{d}_{1,n-i+1}(\lambda^{(2)})=2-i+\overline{n}\geq p-2i+4\geq 5, 
\overline{d}_{1,n-r}(\lambda^{(2)})=p-1,\] 
\[\overline{d}_{1,n-r+1}(\lambda^{(2)})=p+2 \text{ and } \overline{d}_{1,n}(\lambda^{(2)})=i+\overline{n}<2p,\] there is no $d_{1,v}(\lambda^{(2)})\equiv 1 \pmod p$ for 
$n-i<v$. 
By Lemma \ref{l4.6}, we infer $\lambda^{(2)}\in F_0$ and 
\[\lambda^{(2)}\sim \lambda^{(3)}=(a+1)a^{n-i-1}(a-1)\ldots (a-i+r+3)(a-i+r+2)^3(a-i+r-1)\ldots (a-i).\]

Since 
\[\overline{d}_{1,n-i+1}(\lambda^{(3)})=\overline{d}_{1,n-i+1}(\lambda^{(2)})>0,\overline{d}_{1,n-r}(\lambda^{(3)})=p-2,\] 
\[\overline{d}_{1,n-r+1}(\lambda^{(3)})=p+2 \text{ and }
\overline{d}_{1,n}(\lambda^{(3)})=\overline{d}_{1,n}(\lambda^{(2)})<2p,\]
there is no 
$d_{1,v}(\lambda^{(3)})\equiv 1 \pmod p$ for $n-i<v$. By Lemma \ref{l4.6}, we infer $\lambda^{(3)}\in F_0$
and 
 \[\lambda^{(3)}\sim \lambda^{(4)}=(a+1)a^{n-i-1}(a-1)\ldots (a-i+r+2)^3(a-i+r+1)(a-i+r-2)\ldots (a-i).\]

Since 
\[\overline{d}_{1,n-i+1}(\lambda^{(4)})=\overline{d}_{1,n-i+1}(\lambda^{(2)})>0, \overline{d}_{1,n-r+1}(\lambda^{(4)})=p,\]
\[\overline{d}_{1,n-r+2}(\lambda^{(4)})=p+4 \text{ and } \overline{d}_{1,n}(\lambda^{(4)})=\overline{d}_{1,n}(\lambda^{(2)})<2p,\] 
there is no 
$d_{1,v}(\lambda^{(4)})\equiv 1 \pmod p$ for $n-i<v$. By Lemma \ref{l4.6}, we infer $\lambda^{(4)}\in F_0$
and  
\[\begin{aligned}\lambda^{(4)}\sim \lambda^{(5)}=&(a+1)a^{n-i-1}(a-1)\ldots (a-i+r+2)^3(a-i+r+1)(a-i+r)\\
&(a-i+r-3)\ldots (a-i).\end{aligned}\]
For $j=0, \ldots, r-3$ define  
\[\begin{aligned}\mu^{(r-2+j)}=&(a+1)a^{n-i-1}(a-1)\ldots (a-i+r+2)^3(a-i+r+1)(a-i+r)\ldots \\
&\ldots (a-i+r-j)(a-i+r-3-j)\ldots (a-i).\end{aligned}\]
Since 
\[\overline{d}_{1,n-i+1}(\mu^{(r-2+j)})=\overline{d}_{1,n-i+1}(\lambda^{(2)})>0, \overline{d}_{1,n-r+1}(\mu^{(r-2+j)})=p,\] 
\[\overline{d}_{1,n-r+2}(\mu^{(r-2+j)})=p+2 \text{ and } \overline{d}_{1,n}(\mu^{(r-2+j)})=\overline{d}_{1,n}(\lambda^{(2)})<2p,\]
there is no 
$d_{1,v}(\mu^{(r-2+j)})\equiv 1 \pmod p$ for $n-i<v$. By Lemma \ref{l4.6}, we infer $\mu^{(r-2+j)}\in F_0$
and $\mu^{(r-2+j)}\sim \mu^{(r-2+j+1)}$, where 
\[\mu^{(2r-4)}=(a+1)a^{n-i-1}(a-1)\ldots (a-i+r+2)^3(a-i+r+1)(a-i+r)\ldots (a-i+2).\]
Since $\lambda^{(5)}=\mu^{(r-2)}$, we derive 
$\lambda\sim \mu^{(2r-4)}\precneqq \lambda$.

Assume \underline{$\overline{n}-i$ is even}.

Since $\lambda\in F_0$ by Lemma \ref{l4.3}, we have 
\[\lambda\sim (a+1)^2a^{n-i-2}(a-1)\ldots (a-i)=\lambda^{(1)}.\]

We have 
\[\overline{d}_1(\lambda^{(1)})=(2-i+\overline{n},4-i+\overline{n}, \ldots,p-1,p+1,p+3, \ldots, i+\overline{n}),\]
\[\overline{d}_{2, n-i+1}(\lambda^{(1)})=1-i+\overline{n}\geq 1-i+p-i+2=p-2i+3>0,
\overline{d}_{2,n-i+\frac{p-n+i+1}{2}}(\lambda^{(1)})=p,\]
\[\overline{d}_{2,n-i+\frac{p-n+i+1}{2}+1}(\lambda^{(1)})=p+2 \text{ and }
\overline{d}_{2,n}(\lambda^{(1)})=i+\overline{n}-1<2p.\] Therefore there is no $n-i<v$ such that
$d_{2,v}(\lambda^{(1)})\equiv 1 \pmod p$.

Assume $n-i<v$ and write $d_{2,v}(\lambda^{(1)})=c'p^{s'}+b'p^{s'+1}$ for $0<c'<p$. We can assume $b'>0$.
Since $d_{2,v}(\lambda)\not \equiv 1 \pmod p$, we have $c'p^{s'}>1$. 
Then $d_{2, c'p^{s'}+1}(\lambda)=c'p^{s'}$, hence
$(*_{2,v})$ is satisfied by Lemma \ref{new2}.

Assume $n-i<v$, $d_{1,v}(\lambda^{(1)})=cp^s+bp^{s+1}$ for $0<c<p$. We can assume $b>0$.
If $s>0$, then $(*_{1,v})$ is satisfied by Lemma \ref{new}.
If $s=0$, then $\overline{d}_{1,n-i+1}(\lambda^{(1)})=2-i+\overline{n}\geq p-2i+4\geq 3$ and the inspection of the vector $\overline{d}_1(\lambda^{(1)})$ shows that if $n-i<v$, then $c\neq 2$. 
If $c=1$, then we have $d_{1,2}(\lambda^{(1)})=1$, while if $c>2$, then $d_{1,c}(\lambda^{(1)})=c$, hence $(*_{1,v})$ is satisfied by Lemma \ref{new2}.

Therefore by Lemma \ref{l4.4} we infer $\lambda^{(1)}\in F_0$ and 
\[\lambda^{(1)}\sim \lambda^{(2)}=(a+3)(a+1)a^{n-i-2}(a-1)\ldots(a-i).\]

Denote $r=\frac{i-t+1}{2}$. 
Then 
\[d_{1,n-r}(\lambda^{(2)})=3+i-r+n-r-1=2+i-2r+n=
2+i+n-i+t-1=(\underline{n}+1)p+1\] 
and $s_{\epsilon_1-\epsilon_{n-r},(\underline{n}+1)p}\bullet \lambda^{(2)}=\lambda^{(3)}$, where
\[\lambda^{(3)}=(a+2)(a+1)a^{n-i-2}(a-1)\ldots (a-i+r+1)^2(a-i+r-1)\ldots (a-i).\]
Thus, $\lambda^{(2)}\sim_{ev}\lambda^{(3)}$.

We have 
\[\overline{d}_1(\lambda^{(3)})=(3-i+\overline{n},5-i+\overline{n}, \ldots,p-2,p-1,p+2, \ldots, i+\overline{n}+1),\]
where 
$\overline{d}_{1, n-i+1}(\lambda^{(3)})=3-i+\overline{n}\geq 3-i+p-i+2=p-2i+5>1$
and $\overline{d}_{1,n}(\lambda^{(3)})=i+\overline{n}+1<2p$. Therefore there is no $n-i<v$ such that
$d_{1,v}(\lambda^{(3)})\equiv 1 \pmod p$.
We have 
\[\overline{d}_2(\lambda^{(3)})=(1-i+\overline{n},3-i+\overline{n}, \ldots,p-4,p-3,p,p+2, \ldots, i+\overline{n}-1),\]
where 
$\overline{d}_{2, n-i+1}(\lambda^{(3)})=1-i+\overline{n}\geq 1-i+p-i+2=p-2i+3>1$
and $\overline{d}_{2,n}(\lambda^{(3)})=i+\overline{n}-1<2p$. Therefore there is no $n-i<v$ such that
$d_{2,v}(\lambda^{(3)})\equiv 1 \pmod p$.
By Lemma \ref{l4.6}, we infer $\lambda^{(3)}\in F_0$ and
\[\lambda^{(3)}\sim \lambda^{(4)}=(a+2)(a+1)a^{n-i-2}(a-1)\ldots (a-i+r+2)^3(a-i+r-1)\ldots (a-i).\]

Next, we form a sequence of weights 
\[\begin{aligned}\lambda^{(4)}\sim\mu^{(i-r-1)}=&(a+2)(a+1)a^{n-i-2}(a-1)\ldots (a-i+r+2)^3(a-i+r+1)\\
&(a-i+r-2)\ldots (a-i)\end{aligned}\] 
through 
\[\begin{aligned}\mu^{(i-2)}=&(a+2)(a+1)a^{n-i-2}(a-1)\ldots (a-i+r+2)^3(a-i+r+1)(a-i+r)\\
&\ldots (a-i+2)=\kappa\end{aligned}\]
that are all linked because $\overline{d}_{1,v}(\mu^{(j)})=\overline{d}_{1,v}(\lambda^{(4)})$ for $v\leq n-r$, the parities of the remaining entries in $\overline{d}_1(\mu^{(j)})$ and $\overline{d}_2(\mu^{(j)})$ do not change and stay contained within the interval $[2,2p]$, which imply that there is no $n-i<v$ such that
$d_{1,v}(\mu^{(j)})\equiv 1 \pmod p$ and no $n-i<v$ such that $d_{2,v}(\mu^{(j)})\equiv 1 \pmod p$.
Then Lemma \ref{l4.6} implies that all $\mu^{(j)}\in F_0$.

We have 
\[\overline{d}_1(\kappa)=(3-i+\overline{n},5-i+\overline{n}, \ldots,p-6,p-4,p-3,p-2,p,p+2, \ldots, i+\overline{n}-1),\]
and 
\[\overline{d}_2(\kappa)=(1-i+\overline{n},3-i+\overline{n}, \ldots,p-8,p-6,p-5,p-4,p-2, p,\ldots, i+\overline{n}-3),\] 

From $\kappa$, there is a series of linked weights 
\[\kappa^{(j)}=(a+2)(a+1)a^{n-i-2}(a-1)\ldots (a-i+j+2)^3(a-i+j+1)(a-i+j)\ldots (a-i+2)\]
for $j=r, \ldots, i-3$ that belong to $F_0$. 
Indeed, for each weight we have
\[\overline{d}_{1,n-i+1}(\kappa^{(j)})>\overline{d}_{2,n-i+1}(\kappa^{(j)})>1 \text{ and }
\overline{d}_{2,n}(\kappa^{(j)})<\overline{d}_{1,n}(\kappa^{(j)})<2p,\] 
\[\overline{d}_{1,v}(\kappa^{(j)})=\overline{d}_{1,v}(\kappa), \overline{d}_{2,v}(\kappa^{(j)})=\overline{d}_{2,v}(\kappa) \text{ for } v\geq n-r+1,\]
while the remaining entries in $\overline{d}(\kappa^{(j)})$ are less than $p$.
Therefore, there is no $n-i<v$ such that $d_{1,v}(\kappa^{(j)})\equiv 1 \pmod p$ or $d_{2,v}(\kappa^{(j)})\equiv 1 \pmod p$, and Lemma \ref{l4.6} implies $\kappa^{(j)}\in F_0$. 

We obtain 
\[\kappa\sim \kappa^{(i-3)}=(a+2)(a+1)a^{n-i-2}(a-1)^3\ldots (a-2)\ldots (a-i+2)\]
and 
\[\kappa^{(i-3)}\sim (a+2)(a+1)a^{n-i}(a-1)\ldots (a-2)\ldots (a-i+2)=\iota.\]

To finish, we define weights 
\[\iota^{(j)}=(a+2)(a+1)^{1+2j}a^{n-i-2j}(a-1)\ldots (a-2)\ldots (a-i+2)\]
for $j=0, \ldots, \frac{\overline{n}-i}{2}=r$. 

Then for each $j$, 
\[\overline{d}_{1}(\iota^{(j)})=(2+\overline{n}-i, 3+\overline{n}-i,5+\overline{n}-i, \ldots, i+\overline{n}-1 ).\] 
Since 
\[\overline{d}_{1,n-i+3}(\iota^{(j)})\geq 5+\overline{n}-i>1, \quad  
\overline{d}_{1,n}(\iota^{(j)})=i+\overline{n}-1< 2p.\]
and $\overline{d}_{1,v}(\iota^{(j)})$ is odd for $v\geq n-i+3$, 
we infer that there is no $n-i+2<v$ such that $\overline{d}_{1,v}(\iota^{(j)})\equiv 1 \pmod p$.

We have 
\[\overline{d}_{2+2j}(\iota^{(j)})=(-i+\overline{n}-2j, -i+\overline{n}+1-2j,-i+\overline{n}+3-2j, \ldots, i+\overline{n}-1-2j ).\] 
and 
\[\overline{d}_k(\iota^{(j)})=(-i+\overline{n}+2-k, -i+\overline{n}+3-k,-i+\overline{n}+5-k, \ldots, i+\overline{n}+1-k )\] 
for $1<k\leq 2+2j$.

For $1<k\leq 2+2j$ we have 
\[\overline{d}_{k,n-i+3}(\iota^{(j)})=\overline{n}-i+5-k\geq \overline{n}-i+3-2j\geq  3+\overline{n}-i-\overline{n}+i=3\]
and
\[\overline{d}_{k,n}(\iota^{(j)})=i+\overline{n}+1-k<2p.\]
Since for $v\geq n-i+3$ the number $\overline{d}_{2+2j,v}(\iota^{(j)})$ is odd, we derive that $d_{2+2j,v}(\iota^{(j)})\not\equiv 1 \pmod p$ for $v\geq n-i+3$. 

By Lemma \ref{l4.3}, in order to show that $\iota^{(j)}\in F_0$, it is enough to verify 
the condition $(*_{k,v})$ for $1\leq k\leq 2+2j$ and $n-i+2<v$.
Let $d_{k,v}(\lambda)=cp^s+bp^{s+1}$ with $0<c<p$. 

Assume $1<k<2j+2$. If $s>0$, then conditions $(*_{k,v})$ are satisfied by Lemma \ref{new}.
If $s=0$ and $2+2j-k\geq c$, then  $d_{k, k+c}(\iota^{(j)})=c$, while if $s=0$ and $2+2j-k<c$, then 
$d_{k, k+c-1}(\iota^{(j)})=c$. Therefore all conditions $(*_{k,v})$ are satisfied for $1<k<2+2j$ by Lemma \ref{new2}.

Assume $k=1$ or $k=2j+2$. Every condition $(*_{2j+2,v})$ for $v\leq n-i+2$ is clearly satisfied. 
The condition $\overline{n}\geq p-i+2$ implies $i>2$ and $2j+4\leq \overline{n}-i+4\leq p-1-3+4=p$.
Since $d_{1,2j+2}(\iota^{(j)})=2j+2$ and $d_{1,2j+3}(\iota^{(j)})=2j+4\leq p$,
the condition $(*_{1,v})$ is satisfied for $v\leq n-i+2$ by Lemma \ref{l4.8}. Therefore we can assume $n-i+3\leq v$.

Write $d_{2j+2,v}(\iota^{(j)})=c'p^{s'}+b'p^{s'+1}$ with $0<c'<p$. We can assume $b'>0$.
Since $d_{2j+2,v}(\iota^{(j)})\not \equiv 1 \pmod p$, we have $c'p^{s'}>1$. 
If $c'p^{s'}\geq 3$, then 
$d_{2j+2, 2j+c'p^{s'}}(\iota^{(j)})=c'p^{s'}$, while if $c'p^{s'}=2$, then 
$d_{2j+2, 2j+3}(\iota^{(j)})=2$,  hence $(*_{2j+2,v})$ is satisfied by Lemma \ref{new2}.

Write  $d_{1,v}(\iota^{(j)})=cp^s+bp^{s+1}$ with $0<c<p$. We can assume $b>0$.
Since $d_{1,v}(\iota^{(j)})\not \equiv 1 \pmod p$, we have $cp^s>1$.
Since $d_{2j+2,v}(\iota^{(j)})\not \equiv 1 \pmod p$, we have $cp^s\neq 3+2j$. 
If $cp^s\geq 4+2j$, then $d_{1,cp^s-1}=cp^s$, while if $2\leq cp^s\leq 2+2j$, then 
$d_{1,cp^s}(\lambda)=cp^s$. Therefore, condition $(*_{1,v})$ is satisfied by Lemma \ref{new2}.
Using Lemma \ref{l4.4}, we conclude that $\lambda\in F_0$.

Therefore by Lemma \ref{l4.3}, we infer $\iota^{(j)}\in F_0$ for each $j$. 

If $\underline{n}=0$, then 
\[\kappa^{(1)}\sim \iota^{(r)}=(a+2)(a+1)^{n-i+1}(a-1)\ldots (a-i+2)\precneqq \lambda.\]

If $\underline{n}>0$, then $\kappa^{(1)}\sim \iota^{(r)}\sim \iota^{(r+1)}$, where 
\[\iota^{(r+1)}=(a+2)(a+1)^{\overline{n}-i+3}a^{\underline{n}p-2}(a-1)\ldots(a-i+2).\]

Since $d_{\overline{n}-i+4, n-i+3}(\iota^{(r+1)})=\underline{n}p+1$, we have 
$s_{\epsilon_{\overline{n}-i+4}-\epsilon_{n-i+3},\underline{n}p}\bullet \iota^{(r+1)}=\iota^{(r+2)}$, where
\[\iota^{(r+2)}=(a+2)(a+1)^{\overline{n}-i+2}a^{\underline{n}p}(a-2)\ldots (a-i+2).\]
Therefore $\iota^{(r+1)}\sim_{ev}\iota^{(r+2)}\precneqq \lambda$.
\end{proof}

\begin{pr}\label{p6.3}
The weight $\omega^a_{-i}$ for $i\geq 2$ is not minimal if 
$\overline{n}\leq i$.
\end{pr}
\begin{proof}
Write $\lambda=\omega^a_{-i}=a^{n-i}(a-1)\ldots (a-i)$ and $\overline{n}=t$.

If $\underline{n}=0$, then $n\leq i$ and $2i+1\leq p$ gives $n<p$. Therefore, $\underline{n}>0$, which implies  $n\geq p\geq 2i+1\geq i+3$ and shows that all powers of $a$ appearing below are nonnegative.

Assume \underline{$i-t$ is even}.

Since $\lambda\in F_0$ by Lemma \ref{l4.3}, we have 
\[\lambda\sim (a+1)^2a^{n-i-2}(a-1)\ldots (a-i)=\lambda^{(1)}.\]
If $r=\frac{i+t-2}{2}$, then 
\[d_{2,n-r}(\lambda^{(1)})=1+i-r+n-r-2=-1+i+n-i-t+2=\underline{n}p+1\]
and $s_{\epsilon_2-\epsilon_{n-r},\underline{n}p}\bullet \lambda^{(1)}=\lambda^{(2)}$,  where
\[\lambda^{(2)}=(a+1)a^{n-i-1}(a-1)\ldots (a-i+r+2)(a-i+r+1)^2(a-i+r-1)\ldots (a-i),\]
showing that $\lambda^{(1)}\sim_{ev}\lambda^{(2)}$.

Assume $t<i$. Then $r\leq i-2$, $d_{1,n-r}(\lambda^{(2)})=i-r+n-r-1=i+n-1-i-t+2=\underline{n}p+1$
and $s_{\epsilon_1-\epsilon_{n-r},\underline{n}p}\bullet \lambda^{(2)}=\lambda^{(3)}$, 
where
\[\lambda^{(3)}=a^{n-i}(a-1)\ldots (a-i+r+2)^2(a-i+r+1)(a-i+r-1)\ldots (a-i),\]
showing that $\lambda^{(2)}\sim_{ev}\lambda^{(3)}\precneqq \lambda$.

If $t=i$, then $s=i-1$ and $\lambda^{(2)}=(a+1)a^{n-i}(a-2)\ldots (a-i)$.
For $j=2, \ldots i$, define 
\[\mu^{(j)}=(a+1)a^{n-i}a(a-1)\ldots (a-j+3)(a-j+2)(a-j-1)\ldots (a-i).\]
Since  $\overline{d}_{1,n-i+1}(\lambda^{(2)})=1$ and $\overline{d}_{1,n}(\lambda^{(2)})=\overline{n}+i<p$,
Lemma \ref{l4.5} implies $\lambda^{(2)}\in F_0$ and $\lambda^{(2)}\sim \mu^{(2)}$.
Since $\overline{d}_{1,n-i+1}(\mu^{(j)})=1$ and $\overline{d}_{1,n}(\mu^{(j)})=\overline{n}+i<p$ for $j=2, \ldots, i-1$,
Lemma \ref{l4.5} implies $\mu^{(j)}\in F_0$ and $\mu^{(j)}\sim \mu^{(j+1)}$.
Therefore, 
\[\lambda\sim \mu^{(i)}=(a+1)a^{n-i+2}(a-1)\ldots (a-i+2)\precneqq \lambda.\]

Assume \underline{$i-t$ is odd}.
Since $\lambda\in F_0$ by Lemma \ref{l4.3}, we have 
\[\lambda\sim (a+1)^2a^{n-i-2}(a-1)\ldots (a-i)=\lambda^{(1)}.\]

We have 
\[\overline{d}_1(\lambda^{(1)})=(2-i+t,4-i+t, \ldots,-1,1,3, \ldots, i+t),\]
\[\overline{d}_2(\lambda^{(1)})=(1-i+t,3-i+t, \ldots,-2,0,2, \ldots, i+t-1),\]
where 
$\overline{d}_{2, n-i+1}(\lambda^{(1)})=1-i+t\geq 1-i>-p+1$
and $\overline{d}_{2,n}(\lambda^{(1)})=i+t-1<p$. Therefore there is no $n-i<v$ such that
$d_{2,v}(\lambda^{(1)})\equiv 1 \pmod p$, and using the arguments as in the proof of Proposition \ref{p6.2}, we infer that condition $(*_{2,v})$ is satisfied.

Let $d_{1,v}(\lambda^{(1)})=cp^s+bp^{s+1}$ for $0<c<p$.
Since $\overline{d}_{1,n-i+1}(\lambda^{(1)})=2-i+t>-p+1$,  $\overline{d}_{1,n}(\lambda^{(1)})=i+t<p$,
the inspection of the vector $\overline{d}_1(\lambda^{(1)})$ shows that if $s=0$ and $n-i<v$, then $c\neq 2$. 
The same argument as in the proof of Proposition \ref{p6.2} gives that condition $(*_{1,v})$ is satisfied.

Therefore by Lemma \ref{l4.4}, we infer $\lambda^{(1)}\in F_0$ and 
\[\lambda^{(1)}\sim \lambda^{(2)}=(a+2)^2a^{n-i-2}(a-1)\ldots(a-i).\]

Denote $r=\frac{i+t-1}{2}$. 
Then 
\[d_{2,n-r}(\lambda^{(2)})=2+i-r+n-r-2=i+n-2r=
i+n-i-t+1=\underline{n}p+1\] and $s_{\epsilon_2-\epsilon_{n-r},\underline{n}p}\bullet \lambda^{(2)}=\lambda^{(3)}$, where
\[\lambda^{(3)}=(a+2)(a+1)a^{n-i-2}(a-1)\ldots (a-i+r+1)^2(a-i+r-1)\ldots (a-i),\]
showing that $\lambda^{(2)}\sim_{ev}\lambda^{(3)}$.

We have 
\[\overline{d}_1(\lambda^{(3)})=(3-i+t,5-i+t, \ldots,-6,-4,-2,-1,2,4, \ldots, i+t+1),\]
where 
$\overline{d}_{1, n-i+1}(\lambda^{(3)})=3-i+t\geq 3-i>-p+1$
and $\overline{d}_{1,n}(\lambda^{(3)})=i+t+1\leq p$. Therefore there is no $n-i<v$ such that
$d_{1,v}(\lambda^{(3)})\equiv 1 \pmod p$.

We have 
\[\overline{d}_2(\lambda^{(3)})=(1-i+t,3-i+t, \ldots,-8,-6,-4,-3,0,2, \ldots, i+t-1),\]
where 
$\overline{d}_{2, n-i+1}(\lambda^{(3)})=1-i+t\geq 1-i>-p+1$
and $\overline{d}_{2,n}(\lambda^{(3)})=i+t-1<p$. Therefore there is no $n-i<v$ such that
$d_{2,v}(\lambda^{(3)})\equiv 1 \pmod p$.
By Lemma \ref{l4.6} we infer $\lambda^{(3)}\in F_0$ and
\[\lambda^{(3)}\sim \lambda^{(4)}=(a+2)(a+1)a^{n-i-2}(a-1)\ldots (a-i+r+2)^3(a-i+r-1)\ldots (a-i).\]

Next, we form a sequence of weights 
\[\begin{aligned}\lambda^{(4)}\sim\mu^{(i-r-1)}=&(a+2)(a+1)a^{n-i-2}(a-1)\ldots (a-i+r+2)^3(a-i+r+1)\\
&(a-i+r-2)\ldots (a-i)\end{aligned}\] through 
\[\begin{aligned}\mu^{(i-2)}=&(a+2)(a+1)a^{n-i-2}(a-1)\ldots (a-i+r+2)^3(a-i+r+1)(a-i+r)\\
&\ldots (a-i+2)=\kappa\end{aligned}\]
that are all linked because $\overline{d}_{1,v}(\mu^{(j)})=\overline{d}_{1,v}(\lambda^{(4)})$ for $v\leq n-r$, the parities of the remaining entries in $\overline{d}_1(\mu^{(j)})$ and $\overline{d}_2(\mu^{(j)})$ do not change and stay contained within the interval $[-p+2,p]$, which imply that there is no $n-i<v$ such that
$d_{1,v}(\mu^{(j)})\equiv 1 \pmod p$ and no $n-i<v$ such that $d_{2,v}(\mu^{(j)})\equiv 1 \pmod p$. 
Then Lemma \ref{l4.6} implies that all $\mu^{(j)}\in F_0$.

We have 
\[\overline{d}_1(\kappa)=(3-i+t,5-i+t, \ldots,-6,-4,-3,-2,0,2,4, \ldots, i+t-1),\]
and 
\[\overline{d}_2(\kappa)=(1-i+t,3-i+t, \ldots,-8,-6,-5,-4,-2, 0,2,\ldots, i+t-3),\] 

From $\kappa$, there is a series of linked weights 
\[\kappa^{(j)}=(a+2)(a+1)a^{n-i-2}(a-1)\ldots (a-i+j+2)^3(a-i+j+1)(a-i+j)\ldots (a-i+2)\]
for $j=r, \ldots, i-3$ that belong to $F_0$. 

Indeed, for each weight we have
$\overline{d}_{1,n-i+1}(\kappa^{(j)})>\overline{d}_{2,n-i+1}(\kappa^{(j)})>-p+1$ and
$\overline{d}_{2,n}(\kappa^{(j)})<\overline{d}_{1,n}(\kappa^{(j)})< p$ 
and $\overline{d}_{1,v}(\kappa^{(j)})=\overline{d}_{1,v}(\kappa)$, $\overline{d}_{2,v}(\kappa^{(j)})=\overline{d}_{2,v}(\kappa)$ for $v\geq n-r+1$, while the remaining entries in $\overline{d}(\kappa^{(j)})$ are negative.

Therefore, there is no $n-i<v$ such that $d_{1,v}(\kappa^{(j)})\equiv 1 \pmod p$ or $d_{2,v}(\kappa^{(j)})\equiv 1 \pmod p$, and Lemma \ref{l4.6} implies $\kappa^{(j)}\in F_0$, 

We obtain 
\[\kappa\sim \kappa^{(i-3)}=(a+2)(a+1)a^{n-i-2}(a-1)^3\ldots (a-2)\ldots (a-i+2)\]
and 
\[\kappa^{(i-3)}\sim (a+2)(a+1)a^{n-i}(a-1)\ldots (a-2)\ldots (a-i+2)=\iota.\]

To finish, we define weights 
\[\iota^{(j)}=(a+2)(a+1)^{1+2j}a^{n-i-2j}(a-1)\ldots (a-2)\ldots (a-i+2)\]
for $j=0, \ldots, \frac{p+t-i}{2}=r$. 

Then for each $j$,
\[\overline{d}_{1}(\iota^{(j)})=(2+t-i, 3+t-i,5+t-i, \ldots, i+t-1 ).\] 
Since 
\[\overline{d}_{1,n-i+3}(\iota^{(j)})\geq 5+t-i>-p+1, \quad  
\overline{d}_{1,n}(\iota^{(j)})=i+t-1< p.\]
and $\overline{d}_{1,v}(\iota^{(j)})$ is even for $v\geq n-i+3$, 
we infer that there is no $n-i+2<v$ such that $\overline{d}_{1,v}(\iota^{(j)})\equiv 1 \pmod p$.

We have 
\[\overline{d}_{2+2j}(\iota^{(j)})=(-i+t-2j, -i+t+1-2j,-i+t+3-2j, \ldots, i+t-1-2j ).\] 
and 
\[\overline{d}_k(\iota^{(j)})=(-i+t+2-k, -i+t+3-k,-i+t+5-k, \ldots, i+t+1-k )\] 
for $1<k\leq 2+2j$.

For each $j=0,\ldots,r$ and $1<k\leq 2+2j$ we have 
\[\overline{d}_{k,n-i+3}(\iota^{(j)})=t-i+5-k\geq t-i+3-2j> 3+t-i--p+t-i>-p+1\]
and
\[\overline{d}_{k,n}(\iota^{(j)})=i+t+1-k<p.\]
Since for $v\geq n-i+3$ the number $\overline{d}_{2+2j,v}(\iota^{(j)})$ is even, we derive that $d_{2+2j,v}(\iota^{(j)})\not\equiv 1 \pmod p$ for $v\geq n-i+3$.

By Lemma \ref{l4.3}, in order to show that $\iota^{(j)}\in F_0$, it is enough to verify 
the condition $(*_{k,l})$ for $1\leq k\leq 2+2j$ and $n-i+2<v$.

Let $d_{k,l}(\lambda)=cp^s+bp^{s+1}$ with $0<c<p$.
Assume $1<k<2j+2$. 
If $s=0$ and $2+2j-k\geq c$, then  $d_{k, k+c}(\iota^{(j)})=c$. If $2+2j-k<c$, then 
$d_{k, k+c-1}(\iota^{(j)})=c$. Therefore all conditions $(*_{k,l})$ are satisfied for $1<k<2+2j$ by Lemma \ref{new2}. 

Assume $k=1$ or $k=2j+2$. Using the same arguments as in the proof of Proposition \ref{p6.2}, 
we derive that $d_{1,v}(\iota^{j}))\not\equiv 1 \pmod p$ and 
$d_{2j+2,v}(\iota^{j}))\not\equiv 1 \pmod p$ for $v\geq n-i+3$ implies that conditions $(*_{1,v})$ and $(*_{2j+2,v})$ 
are satisfied.

Therefore by Lemma \ref{l4.6} we infer $\iota^{(j)}\in F_0$ for each $0\leq j\leq r$. 

Since $\underline{n}>0$, we get 
\[\kappa^{(1)}\sim \iota^{(r)}=(a+2)(a+1)^{p+t-i+1}a^{(\underline{n}-1)p}(a-1)\ldots(a-i+2).\]

If $\underline{n}=1$, then $\iota^{(r)}\precneqq \lambda$.

If $\underline{n}>1$, then 
\[\iota^{(r)}\sim\iota^{(r+1)}=(a+2)(a+1)^{p+t-i+3}a^{(\underline{n}-1)p-2}(a-1)\ldots(a-i+2).\]
Since \[d_{p+t-i+4, n-i+3}(\iota^{(r+1)})=n-p-t-1+2=(\underline{n}-1)p+1,\]
we have $s_{\epsilon_{p+t-i+4}-\epsilon_{n-i+3}(\underline{n}-1),p}\bullet \iota^{(r+1)}=\iota^{(r+2)}$,
where
\[\iota^{(r+2)}=(a+2)(a+1)^{p+t-i+2}a^{(\underline{n}-1)p}(a-2)\ldots (a-i+2).\]
We conclude that $\iota^{(r)}\sim_{ev}\iota^{(r+1)}\precneqq \lambda$.
\end{proof}

\section{$\omega^a_{-1}$ and $\omega^a_0$}

\begin{lm}\label{l7.1}
The weights $(a+1)^ta^{n-t}$ for $1\leq t<n$ belong to $F_0$.
\end{lm}
\begin{proof}
Let $\lambda=(a+1)^ta^{n-t}$. Take $1\leq u<v\leq n$ and write $d_{u,v}(\lambda)=cp^s+bp^{s+1}$.
Denote by $\alpha_{u,j}$ the minimal integer such that $d_{u,\alpha_{u,j}}(\lambda)\geq jp^{s+1}$. 
If $d_{u,\alpha_{u,j}}(\lambda)> jp^{s+1}$, then $d_{u,\alpha_{u,j}}(\lambda)= jp^{s+1}+1$, 
$\alpha_{u,j}=t+1$ and $d_{u,t}(\lambda)=t-u=jp^{s+1}-1$.
We have a sequence $u=i_0<i_1<\ldots <i_b<i_{b+1}=v$ such that 
$i_1=u+cp^s, \ldots, i_{j}=u+cp^s+(j-1)p^{s+1}\leq u+jp^{s+1}-1=t$,
$i_{j+1}= u+cp^s+jp^{s+1}-1=t+cp^s\geq t+1$, \ldots, 
$i_{b+1}=u+cp^s+bj^{s+1}-1=v$ 
for which $d_{i_0, i_1}(\lambda)=cp^s$, and $d_{i_r,i_{r+1}}(\lambda)=p^{s+1}$ for each $r=1, \ldots, b$,
showing that $\lambda\in F_0$.
\end{proof}

\begin{lm}\label{l7.2}
In $n$ is even, then $\omega^a_0\sim \omega^{a+1}_0$ and $\omega^a_{-1}\sim \omega^{a+1}_{-1}$.

If $n$ is odd, then $\omega^a_{-1}\sim\omega^{a-1}_0$ and $\omega^{a-1}_0\sim \omega^{a+1}_0$.
\end{lm}
\begin{proof}
If \underline{$n$ is even},  then 
$\omega^a_0=a^n\sim(a+1)^2a^{n-2}\sim \ldots \sim (a+1)^{2u}a^{n-2u}\sim\ldots \sim (a+1)^n=\omega^{a+1}_0$ 
because $(a+1)^{2u}a^{n-2u}\in F_0$ for $0\leq u<\frac{n}{2}$ by Lemma \ref{l7.1}. 

If $n$ is odd, then $\omega^a_{-1}=a^{n-1}(a-1)\sim a^{n-3}(a-1)^3\sim \ldots \sim a^{n-1-2v}(a-1)^{2v+1}\sim \ldots \sim (a-1)^n=\omega^{a-1}_0$  
because $a^{n-1-2v}(a-1)^{2v+1}\in F_0$ for $0\leq v<\frac{n-1}{2}$ by Lemma \ref{l7.1}.

Assume $n$ is even and $p\geq n$. 
Then
$\omega^a_{-1}=a^{n-1}(a-1)\sim (a+2)a^{n-2}(a-1)\sim \ldots \sim (a+2j)a^{n-2}(a-1)\sim \ldots \sim 
(a+p-n+1)a^{n-2}(a-1)$ because each $(a+2j)a^{n-2}(a-1)\in F_0$ for $0\leq j< \frac{p-n+1}{2}$ by Lemma \ref{l5.1}.
Let $\mu^{(1)}=(a+p-n+1)a^{n-2}(a-1)$. Then $d_{1n}(\mu^{(1)})=p-n+2+n-1=p+1$ and 
$s_{\epsilon_1-\epsilon_{n},p}\bullet \mu^{(1)}=\mu^{(2)}$, where
$\mu^{(2)}=(a+p-n)a^{n-1}$.      
Therefore, $\mu^{(1)}\sim_{ev}\mu^{(2)}$. 
Also, 
$\mu^{(2)}\in F_0$ by Lemma \ref{l5.1}.
Then $(a+p-n)a^{n-1}\sim (a+p-n-2)a^{n-1}\sim \ldots \sim (a+p-n-2j)a^{n-1}\sim \ldots \sim (a+1)a^{n-1}$
because each $(a+p-n-2j)a^{n-1}\in F_0$ by Lemma \ref{l5.1}.
Then  $(a+1)a^{n-1}\sim (a+1)^{n-1}a=\omega^{a+1}_{-1}$ by Lemma \ref{l7.1}, and we conclude
$\omega^a_{-1}\sim \omega^{a+1}_{-1}$.

If \underline{$n$ is odd} and $p>n$,  then $\omega^a_{0}\sim (a+2)a^{n-1}\sim\ldots \sim (a+2)^ja^{n-j}\sim \ldots (a+2)^n=\omega^{a+2}_0$ because each weight $(a+2)^ja^{n-j}\in F_0$ by Lemma \ref{l5.1}.

Assume $p< n$, $\lambda=\omega^a_{-1}$ and $\overline{n}$ is odd. For $j=0, \ldots, r=\frac{\overline{n}+1}{2}$ denote by $\lambda^{(j)}=(a+1)^{2j} a^{n-2j-1}(a-1)$. Then $d_{1,2j}(\lambda^{(j)})=2j \leq 2r=\overline{n}+1<p$ and 
$d_{2j,n-1}(\lambda^{(j)})=n-1-2j+2j=n-1\geq p$ imply that for every $1\leq k\leq 2j$ either $s>0$ or there is an index $2j<t\leq n-1$ such that 
$d_{k,t}=c$ which implies the condition $(*_{k,v})$ is satisfied for every $1\leq k\leq 2j$ and $k<v$. Thus by Lemma \ref{l4.4}, we have 
$\lambda^{(j)}\in F_0$.
Since $d_{2r,n}(\lambda^{(r)})=n-2r+2=n-\overline{n}-1+2=\underline{n}p+1$, we have 
$s_{\epsilon_{2r}-\epsilon_n,\underline{n}p}\bullet\lambda^{(r)}=\lambda^{(r+1)}$, 
where $\lambda^{(r+1)}=(a+1)^{2r-1} a^{n-2r+1}$. 
For $j=0, \ldots, r-1$ define $\kappa^{(j)}=(a+1)^{2j+1}a^{n-2j-1}$. By Lemma \ref{l7.1}, each $\lambda^{(r+1)}\in F_0$. Therefore $\omega^a_{-1}\sim \lambda^{(r+1)}=\kappa^{(r-1)}\sim \kappa^{(0)}=(a+1)a^{n-1}$.

Assume $p< n$, $\lambda=\omega^a_{-1}$ and $\overline{n}$ is even.
Define 
\[\lambda^{(j)}=(a+2j)a^{n-2}(a-1)\]
 for $j=1, \ldots, \frac{p+1-\overline{n}}{2}=r$. 
Then  for $j<r$ there is 
$\overline{d}_{1,n-1}(\lambda^{(j)})=-2+\overline{n}+2j\geq 0$ and 
$\overline{d}_{1,n}(\lambda^{(j)})=\overline{n}+2j\leq \overline{n}+p-1-\overline{n}<p$. Therefore by Lemma \ref{l4.5}, we have  $\lambda^{(j)}\in F_0$ for $1\leq j<r$.

Thus $\lambda\sim \lambda^{(r)}$ and 
$d_{1n}(\lambda^{(r)})=(\underline{n}+1)p+1$, which implies 
$s_{\epsilon_1-\epsilon_{n},(\underline{n}+1)p}\bullet \lambda^{(r)}=\lambda^{(r+1)}$, 
where
\[\lambda^{(r+1)}=(a+2r-1)a^{n-1},\]
hence $\lambda^{(r)}\sim_{ev} \lambda^{(r+1)}$.
Define \[\mu^{(t)}=(a-1+2t)a^{n-1}\] for $t=1, \ldots, r$. Then all $\mu^{(t)}\in F_0$. Since $\mu^{(r)}=\lambda^{(r+1)}$, 
$\lambda\sim\lambda^{(r+1)}\sim\mu^{(1)}=(a+1)a^{n-1}$.

Using Lemma \ref{l7.1} we obtain 
$(a+1)a^{n-1}\sim (a+1)^{n-1}a=\omega^{a+1}_{-1}$ if $n$ is even, and
$(a+1)a^{n-1}\sim (a+1)^n=\omega^{a+1}_0$ if $n$ is odd.
If $n$ is odd,  we combine $\omega^{a-1}_0\sim \omega^a_{-1}\sim \omega^{a+1}_0$. 
\end{proof}

Since $|\alpha|=0$ for every even root $\alpha$, and $|\alpha|=\pm 2$ for every odd root $\alpha$ of $\mathcal{P} (n)$, the category $\mathcal{F}$ of finite-dimensional supermodules over ${\bf P}(n)$ splits as $\mathcal{F}=\mathcal{F}_0\oplus \mathcal{F}_1$, where $\mathcal{F}_0$ (and $\mathcal{F}_1$, respectively) consists of all supermodules $M$ such that if the weightspace $M_{\lambda}\neq 0$, then $|\lambda|$ is even (odd, respectively).

Additionally, if $M\in \mathcal{F}$ is indecomposable, $0\neq v_{\lambda}\in M_{\lambda}$ and $0\neq v_{\mu}\in M_{\mu}$, then $|v_{\lambda}|-|v_{\mu}|\equiv \frac{|\lambda|-|\mu|}{2} \pmod 2$.
Therefore, if $M\in \mathcal{F}_0$, then the parities of $v_{\lambda}-\frac{|\lambda|}{2}$ and $|v_{\mu}|-\frac{|\mu|}{2}$ are the same. Let $\mathcal{F}_{00}$ and $\mathcal{F}_{01}$ consists of those supermodules in $\mathcal{F}_0$ 
for which the above expression $|v_{\lambda}|-\frac{|\lambda|}{2}$ are even or odd, respectively. 
Then $\mathcal{F}_0=\mathcal{F}_{00}\oplus \mathcal{F}_{01}$.  
Analogously, if $M\in \mathcal{F}_1$, then the parities of $v_{\lambda}-\frac{|\lambda|-1}{2}$ and $|v_{\mu}|-\frac{|\mu|-1}{2}$ are the same. Let $\mathcal{F}_{10}$ and $\mathcal{F}_{11}$ consists of those supermodules in $\mathcal{F}_1$ 
for which the above expression $|v_{\lambda}|-\frac{|\lambda|-1}{2}$ are even or odd, respectively. 
Then $\mathcal{F}_1=\mathcal{F}_{10}\oplus \mathcal{F}_{11}$.  

We conclude that $\mathcal{F}$ splits as $\mathcal{F}=\mathcal{F}_{00}\oplus \mathcal{F}_{01}\oplus\mathcal{F}_{10}\oplus \mathcal{F}_{11}$.

\begin{theorem} \label{t7.3}
${\bf P} (n)$ has four blocks for any $p>2$. They are represented by $L^{\epsilon}(\omega_{-1})$ and
$L^{\epsilon}(\omega_0)$, where $\omega_0=\omega^{0}_0$, $\omega_{-1}=\omega^{0}_{-1}$
and $\epsilon\in\{0,1\}$.
\end{theorem}
\begin{proof}
From Propositions \ref{p3.3}, \ref{p3.4}, \ref{p5.5}, \ref{p6.1}, \ref{p6.2}, \ref{p6.3} and Lemma \ref{l7.2}
we derive that each $\lambda$ is linked to $L^{\epsilon}(\omega_0)$ or $L^{\epsilon}(\omega_1)$, where $\epsilon\in\{0,1\}$. Since $L^0(\omega_0)\in \mathcal{F}_{00}$, $L^1(\omega_0)\in \mathcal{F}_{01}$, 
$L^0(\omega_1)\in \mathcal{F}_{11}$ and $L^1(\omega_1)\in \mathcal{F}_{10}$ , they belong to different blocks.
\end{proof}

\end{document}